\documentclass[11pt]{amsart} %

\usepackage{float}

\usepackage{epsfig}
\usepackage{epstopdf}
\usepackage{tikz}
\usepackage{array,amsmath, enumerate,  url, psfrag}
\usepackage{amssymb, amsaddr, fullpage}
\usepackage{graphicx,subfigure}
\usepackage{color}
\usepackage{cite}

\usepackage{appendix}

\theoremstyle{plain}
\newtheorem{thm}{Theorem}[section]

\newtheorem{claim}[thm]{Claim}

\numberwithin{theorem}{section}

\numberwithin{lemma}{section}

\numberwithin{corollary}{section}

\numberwithin{conjecture}{section}
\newtheorem*{conjecture*}{Conjecture}

\newtheorem{conclusion}{Conclusion}

\title{Sufficient conditions for 2-dimensional global rigidity}

\author{Xiaofeng Gu$^{1}$, Wei Meng$^{2}$,  Martin Rolek$^{3}$,  Yue Wang$^{4}$,  Gexin Yu$^{5}$}

\address{
$^{1}$ Department of Mathematics, University of West Georgia, Carrollton, GA, USA. \\
$^{2}$ School of Mathematical Sciences, Shanxi University, Taiyuan, Shanxi, China.\\
$^{3}$ Department of Mathematics, Kennesaw State University, Marietta, GA, USA.\\
$^{4}$ School of Mathematics, Shandong University, Jinan, Shandong, China.\\
$^{5}$ Department of Mathematics, William \& Mary, Williamsburg, VA, USA.}

\thanks{Research of the first author was partially supported by a grant from the Simons Foundation (522728). The work was done while the second and fourth author were at William \& Mary as visiting scholars. The second author is partially supported by Shanxi University and the National Natural Sciences Foundation for Young Scientists of China (11701349). The fourth author is partially  supported by the Chinese Scholarship Council.  The research of the last author was supported in part by a summer research grant from William \& Mary.}

\email{xgu@westga.edu, mengwei@sxu.edu.cn, mrolek1@kennesaw.edu, m15064013175@163.com,\\
 gyu@wm.edu}

\begin{document}
\date{}
\maketitle

\begin{abstract}
The 2-dimensional global rigidity has been shown to be equivalent to 3-connectedness and redundant rigidity by a combination of two results due to Jackson and Jord\'an, and Connelly, respectively. By the characterization, a theorem of  Lov\'asz and Yemini implies that every $6$-connected graph is redundantly rigid, and thus globally rigid. The 6-connectedness is best possible, since there exist infinitely many 5-connected non-rigid graphs. Jackson, Servatius and Servatius used the idea of ``essential connectivity'' and proved that every 4-connected ``essentially 6-connected'' graph is redundantly rigid and thus global rigid. Since 3-connectedness is a necessary condition of global rigidity, it is interesting to study 3-connected graphs for redundant rigidity and thus globally rigidity.  We utilize a different ``essential connectivity'', and  prove that every 3-connected essentially 9-connected graph is redundantly rigid and thus globally rigid. The essential 9-connectedness is best possible.  Under this essential connectivity, we also prove that every 4-connected essentially 6-connected graph is redundantly rigid and thus global rigid.  Our proofs are based on discharging arguments.
\end{abstract}

\section{Introduction}
Undefined graph terminologies can be found in \cite{BoMu08}.
A {\bf $d$-dimensional framework} is a pair $(G, p)$, where $G(V, E)$ is a graph and
$p$ is a map from $V$ to $\mathbb{R}^d$. Roughly speaking, it is a straight line realization of $G$ in $\mathbb{R}^d$. Two frameworks $(G, p)$ and $(G, q)$ are {\bf equivalent} if $||p(u) - p(v) || = ||q(u) - q(v) ||$ holds for every edge $uv\in E$, where $||\cdot ||$ denotes the Euclidean norm in $\mathbb{R}^d$. Two frameworks $(G, p)$ and $(G, q)$ are {\bf congruent} if $||p(u) - p(v) || = ||q(u) - q(v) ||$ holds for every pair $u, v\in V$. A framework $(G,p)$ is {\bf generic} if the coordinates of all the points are algebraically independent.
A graph $G$ is {\bf globally rigid} if every framework $(G, q)$ which is equivalent to a generic framework $(G, p)$ is congruent to $(G, p)$.
A graph is {\bf rigid} if for every generic framework $(G, p)$ there exists an $\varepsilon >0$ such that if $(G, q)$ is equivalent to $(G, p)$ and $||p(u) - q(u) || < \varepsilon$ for every $u\in V$, then $(G, q)$ is congruent to $(G, p)$.
A graph $G$ is {\bf minimally rigid} if $G$ is rigid and $G-e$ is not rigid for all $e\in E$.
Laman~\cite{Lama70} provides a combinatorial characterization of minimally rigid graphs in $\mathbb{R}^2$, given below.

Let $G$ be a graph with vertex set $V(G)$ and edge set $E(G)$. For a subset $X\subseteq V(G)$, $G[X]$ and $E(X)$ denote the subgraph of $G$ induced by $X$ and the edge set of $G[X]$, respectively. A graph $G$ is {\bf sparse} if $|E(X)|\le 2|X|-3$ for every $X\subseteq V(G)$
with $|X|\ge 2$. If in addition $|E(G)|=2|V(G)|-3$, then $G$ is {\bf minimally rigid}.
In history, Pollaczek-Geiringer \cite{Poll1927,Poll1932} was the first who made notable progress on properties of
minimally rigid graphs. Laman \cite{Lama70} rediscovered and  characterized the minimally rigid graphs in $\mathbb{R}^2$
using the edge count property.   A minimally rigid graph is also known as a {\bf Laman graph} now.

By definition, any sparse graph is simple. A graph $G$ is {\bf rigid} if $G$ contains a spanning minimally rigid subgraph. It is not hard to see that every rigid graph with at least 3 vertices is $2$-connected.
A {\bf cover} of a graph $G$ is a collection $\mathcal{X}=\{X_1,X_2,...,X_t\}$ of subsets of $V(G)$ such that $E(G) = E(X_1)\cup E(X_2)\cup \cdots \cup E(X_t)$. Lov\'asz and Yemini \cite{LoYe82} obtained the following characterization of rigid graphs.

\begin{thm}[Lov\'asz and Yemini \cite{LoYe82}]\label{th1}
Let $G=(V,E)$ be a graph. Then $G$ is rigid if and only if for all covers $\mathcal{X}$ of $G$, we have $\sum_{X\in\mathcal{X}}(2|X|-3)\geq 2|V|-3$.
\end{thm}
Note that the subgraphs induced by $X\in\mathcal{X}$ in Theorem~\ref{th1} need not to be edge-disjoint. However, it is not hard to see the minimum of $\sum_{X\in\mathcal{X}}(2|X|-3)$ can be obtained when the subgraphs induced by $X\in\mathcal{X}$ are pairwise edge-disjoint, and thus it suffices to consider the edge-disjoint case whenever needed.

A graph $G$ is {\bf redundantly rigid} if $G-e$ is rigid for all $e\in E(G)$.  Hendrickson~\cite{Hend92} showed that if a graph $G$ is globally rigid, then $G$ is 3-connected and redundantly rigid. He also conjectured that 3-connectedness and redundant rigidity are sufficient for global rigidity. This conjecture was solved by Connelly~\cite{Conn05}, Jackson and Jord\'an~\cite{JaJo05}, respectively.
\begin{thm}[Connelly~\cite{Conn05}, Jackson and Jord\'an~\cite{JaJo05}]
\label{globthm}
A graph $G$ is globally rigid if and only if $G$ is 3-connected and redundantly rigid, or $G$ is a complete graph on at most three vertices.
\end{thm}

By way of Theorem~\ref{th1}, Lov\'asz and Yemini \cite{LoYe82} showed that every $6$-connected graph $G$ is rigid.  Their proof actually implies that $G$ is redundantly rigid. Thus, every 6-connected graph is globally rigid, by Theorem~\ref{globthm}. The 6-connectedness is best possible, and Lov\'asz and Yemini \cite{LoYe82} constructed infinitely many 5-connected non-rigid graphs.

After that, researchers tried to relax the connectivity condition. For example, Jackson and Jord\'{a}n \cite{JaJo09} proved that a simple graph G is rigid if $G$ is 6-edge-connected, $G-v$ is 4-edge-connected for every  $v\in V (G)$ and $G-\{u, v\}$ is 2-edge-connected for every $u, v\in V (G)$. Jackson, Servatius and Servatius~\cite{JaSS07} provided another connectivity condition sufficient for global rigidity.   They showed that every ``essentially 6-connected'' graph is redundantly rigid, and thus is global rigid,  where a graph $G$ is  ``essentially 6-connected'' if it satisfies:
\begin{enumerate}
\item  $G$ is 4-connected,
\item  for all pairs of subgraphs $G_1,G_2$ of $G$ such that $G=G_1\cup G_2$, $|V(G_1)-V(G_2)|\ge3$ and $|V(G_2)-V(G_1)|\ge3$, we have $|V(G_1)\cap V(G_2)|\ge5$, and
\item  for all pairs of subgraphs $G_1,G_2$ of $G$ such that $G=G_1\cup G_2$, $|V(G_1)-V(G_2)|\ge4$ and $|V(G_2)-V(G_1)|\ge4$, we have $|V(G_1)\cap V(G_2)|\ge6$.
\end{enumerate}

This connectivity allows cuts of size $4$ or $5$, which can only separate at most two or three vertices from the rest of graphs.
In fact, various ideas of ``essential connectivity'' have been used for research. 
In \cite{JaWo92}, Jackson and Wormald defined the ``essentially 4-connected'' graphs to study longest cycles, where a graph $G$ is
essentially 4-connected if $G$ is 3-connected and for every 3-cut $S$ of $G$, $G-S$ has exactly two components, one of which is a single vertex.
To study Hamiltonian claw-free graphs, Lai et al.~\cite{LSWZ06} defined another version of ``essentially $k$-connectivity''.
In their paper, a graph $G$ with at least $k+1$ vertices is {\bf essentially $k$-connected} if there is no $X\subset V(G)$ with $|X|< k$ such that at least
two components of $V-X$ are nontrivial, where a nontrivial component means it contains at least one edge.
Using this definition, Lai et al.~\cite{LSWZ06} showed that every 3-connected essentially 11-connected claw-free graph is Hamiltonian,
which extends Thomassen's conjecture on line Hamiltonian graphs.

We need to point out that the definitions in \cite{JaSS07, JaWo92, LSWZ06} are not the same.
Apparently, one main difference is that the definition of \cite{LSWZ06} requires nontrivial components for cuts of small size,
i.e., components with at least one edge. This requirement allows many vertices of low degrees in graphs with high essential connectivity.
For example, for 3-connected essentially 9-connected graphs, the definition allows this graph contains many vertices of degrees 3 and 4.
Without this requirement, the number of low degree vertices is much more restricted.
In fact, there are many graphs failing to be high connected due to vertices of low degree, and we would like to study such graphs.
Thus, in this paper, we will use essentially $k$-connected graphs defined in \cite{LSWZ06}.

\medskip
The study of rigidity becomes more involved when we work on this essential connectivity.  One may see that this essential connectivity is not monotone, that is,  we may decrease the essential connectivity when adding edges to a graph. The proofs in \cite{LoYe82} and \cite{JaSS07} utilized useful facts that each $X$ in the cover induces a clique and every vertex lies in at least two $X$'s, which rely on the benefit that adding edges does not decrease the (essential) connectivity.

Overcome this difficulty, we first extend the result of Lov\'asz and Yemini \cite{LoYe82}.
The following theorem is quite similar to the main result of \cite{JaSS07}, but they are not the same.
For instance, the result of \cite{JaSS07} fails for a graph containing three or more vertices of degree 4  with the same neighbors,
while Theorem~\ref{th4ess6} allows many such vertices.

\begin{thm} \label{th4ess6}
Every 4-connected essentially 6-connected graph is redundantly rigid, and thus is globally rigid.
\end{thm}
The 5-connected non-rigid graphs constructed by Lov\'asz and Yemini \cite{LoYe82} show that the essential 6-connectedness in Theorem~\ref{th4ess6} is also best possible.
\medskip

Lov\'asz and Yemini \cite{LoYe82} showed that every 6-connected graph is redundantly rigid, which implies that the graph is also globally rigid. It is natural to ask, if the connectivity of a graph is less than 6, what condition can make this graph to be redundantly rigid, and in addition, global rigid?  It was answered by \cite{JaSS07} and Theorem~\ref{th4ess6} for 4-connected graphs. We provide an answer for 3-connected graphs  by giving the optimal essential connectivity.  Recall that 3-connectedness is a necessary condition for global rigidity.

\begin{thm} \label{th3ess9}
Every 3-connected essentially 9-connected graph is redundantly rigid, and thus is globally rigid.
\end{thm}

The proofs of Theorems~\ref{th4ess6} and ~\ref{th3ess9}  utilize discharging arguments, which does not appear explicitly in the literature, as far as we know. Note that Theorems~\ref{th3ess9} and \ref{th4ess6} are also true for multigraphs, since the removal of any multiple edges does not affect connectivity or essential connectivity.

The paper is organized as below.  We prove Theorems~\ref{th4ess6} and ~\ref{th3ess9} respectively in the next two sections. In the last section, for every $3\le t\le 8$, we present examples of $3$-connected essentially $t$-connected non-rigid graphs.

For a vertex $v\in V(G)$, let $d(v)$ and $N(v)$ denote the degree of $v$ and the set of vertices adjacent to $v$, respectively.  A vertex $v$ is called a $k$-vertex, $k^{+}$-vertex and $k^{-}$-vertex, if $d(v)=k$, $d(v)\geq k$ and $d(v)\leq k$, respectively. For $X, Y\subset V(G)$ with $X\cap Y =\emptyset$, let $E[X, Y]$ denote the set of edges in $G$ with one end in $X$ and the other in $Y$ and $e(X, Y)=|E[X,Y]|$.


\section{The proof of Theorem \ref{th4ess6}}

Suppose to the contrary that the statement is not true. Let $G=(V, E)$ be a counterexample such that $|V(G)|$ is as small as possible, and subject to this condition, $|E(G)|$ is as large as possible.
Thus $G$ is not redundantly rigid, and there exists an edge $e=uv$ such that $T=G-e$ is not rigid.
By Theorem~\ref{th1}, there exists a cover $\mathcal{T}=\{Y_1,Y_2,...,Y_t\}$ such that $\sum_{i=1}^t (2|Y_i|-3) < 2|V|-3$.
Without loss of generality, we may assume that $\mathcal{T}=\{Y_1,Y_2,...,Y_t\}$ is a cover of $T$ to minimize $\sum_{i=1}^t (2|Y_i|-3)$.
Let $m=t+1$, $X_i = Y_i$ for $i=1,2, \cdots, t$ and $X_m = X_{t+1} = \{u, v\}$.
Let $\mathcal{X}=\{X_1,X_2,\ldots, X_m\}$.  Then $\mathcal{X}$ is a cover of $G$.  For each $X\in \mathcal{X}$, let $\mu(X)=2|X|-3$.
We have
\begin{equation}
\label{46mu}
\sum_{i=1}^{m} \mu(X_i)=\sum_{i=1}^{m}(2|X_i|-3) = \sum_{i=1}^{t}(2|Y_i|-3) + (2|X_m| -3) < 2|V| -2.
\end{equation}

Let $\mathcal{X}(x)=\{X_i\in\mathcal{X}: x\in X_i\}$ for $x\in V(G)$. For $X\in \mathcal{X}(x)$, we denote by $\sigma_x(X)$ the value that $x$ gets from $X$. We use the following discharging rules.
	
\begin{enumerate}[(R1)]
\item If $|X|=2$, then $\sigma_{x}(X)=\frac12$ for each $x\in X$.
\item If $|X|=3$, then $\sigma_{x}(X)=1$ for each $x\in X$.
\item Suppose $|X|=4$.  If $X$ contains a 4-vertex $x$ such that $\mathcal{X}(x)=\{X_1,X_2\}$ with $|X_1|=2,|X_2|=4$, then $\sigma_{x}(X)=\frac{3}{2}$ and each of other vertices gets 1 from $X$. Otherwise, each vertex of $X$ gets $\frac54$ from $X$.

\item Suppose $|X|\geq5$. 
\begin{enumerate}
\item [(R4A)] If $N(x)\cup \{x\}\subseteq X$, then $\sigma_{x}(X)=2$. Let $V_0 (X)$ be the set of all those vertices $x$ in $X$, and let $|V_0(X)|=t_0(X)$.
\item[(R4B)] If $x\in X$ has only one neighbor $x'\notin X$ and there is an $X'=\{x,x'\}\in \mathcal{X}$, then $\sigma_{x}(X)=\frac32$. Let $V_1(X)$ be the set of all those vertices $x$ in $X$, and denote $|V_1(X)|=t_1(X)$. Let $V'_1(X)=\{x'\in N(x)\backslash X: x\in V_1(X)\}$.
\item[(R4C)] Otherwise, $\sigma_{x}(X)=1$. Let $V_2 (X)$ be the set of all those vertices $x$ in $X$, and denote $|V_2(X)|=t_2(X)$. Let $V'_2(X)=\{x'\in N(x)\backslash X: x\in V_2(X)\}$.
\end{enumerate}
\end{enumerate}

For $x\in V(G)$, let $\mu(x)=\sum_{X\in \mathcal{X}(x)}\sigma_x(X)$, i.e., the total charge that $x$ gets from all $X\in\mathcal{X}(x)$.

\begin{claim}\label{46c1}
$\mu(x)\geq2$, for every vertex $x\in V(G)$.
\end{claim}
\begin{proof}
First let $d(x)=4$. If $|\mathcal{X}(x)|=4$, then $x$ is in four sets of size at least two, thus $\mu(x)\geq\frac12\cdot 4=2$ by (R1)-(R4).   If $|\mathcal{X}(x)|=3$, then $x$ is in at least one set of size at least three, thus $\mu(x)\geq\frac12\cdot 2+1=2$ by (R1)-(R4).
If $|\mathcal{X}(x)|=2$, then $x$ is either in two sets of size at least three, thus by (R2)-(R4), $\mu(x)\geq1+1=2$,  or $x$ is in a set of size two and in a set of size at least four, thus by (R1), (R3), and (R4B), $\mu(x)\geq\frac12+\min\{2,\frac32\}=2$.

Let $d(x)\geq5$. By (R4A), we may assume that $|\mathcal{X}(x)|\geq2$. If $|\mathcal{X}(x)|\geq4$, then $x$ is in four sets of size at least two, thus $\mu(x)\geq\frac12\cdot 4=2$ by (R1)-(R4). If $x$ is in at least two sets of size at least three, then $\mu(x)\geq1+1=2$ by (R2)-(R4). So we may assume that $x$ is either in a set of size two and a set of size at least five, or in two sets of size two and one set of size at least four. Thus, in the former case, by (R1) and (R4B), $\mu(x)=\frac12+\frac32=2$. In the latter case, by (R1) and (R3), $\mu(x)=\frac12\cdot 2+1=2$.
\end{proof}

Now we will try to prove that for each $X\in \mathcal{X}$,  the final charge $\mu'(X)\ge 0$.

\begin{claim}
For each $X\in \mathcal{X}$ with $|X|\le 4$, $\mu'(X)\ge 0$.
\end{claim}

\begin{proof}
If $|X|=2$, then $\mu'(X)=2\cdot2-3-\frac12\cdot2=0$ by (R1).
If $|X|=3$, then $\mu'(X)=2\cdot3-3-1\cdot3=0$ by (R2).
If $|X|=4$, then $\mu(X)=5$. We will show that $X$ contains at most two 4-vertices such that $\mathcal{X}(x)=\{X_1,X_2\}$ with $|X_1|=2,|X_2|=4$. Thus, $\mu'(X)\geq 5-\max\{\frac32\cdot2+1\cdot2,\frac54\cdot4\}=0$.

Suppose to the contrary that $X=\{v_i: i\in[4]\}$ contains three such vertices $v_1,v_2,v_3$ and let $u_i\notin X$ denote the neighbor of $v_i$, for $i=1,2,3$. Note that $X$ induces a clique. If $V(G)-X-\{u_1,u_2,u_3\}$ induces at least one edge $e'$, then $S=\{u_1,u_2,u_3,v_4\}$ separates $\{v_1,v_2,v_3\}$ and $e'$. We obtain an essential 4-cut $S$, a contradiction. Thus we may assume that $V(G)-X-\{u_1,u_2,u_3\}$ is an independent set or an empty set. In the former case, since $G$ is 4-connected, every vertex in $V(G)-X-\{u_1,u_2,u_3\}$ is adjacent to every vertex in $\{u_1, u_2, u_3, v_4\}$. Now $S'=\{v_1,v_4,u_2,u_3\}$ separates $\{v_2,v_3\}$ from $V(G)-X-\{u_2,u_3\}$, and is an essential 4-cut, a contradiction. In the latter case, $V(G)=X\cup\{u_1,u_2,u_3\}$ with $|V(G)|=7$ and $|E(G)|\geq \frac{4\cdot7}{2}=14$. Since $X$ covers 6 edges and $X_i=\{v_i,u_i\}$ for $i\in[3]$ covers one edge, there are more sets in $\mathcal{X}$ covering the other $|E(G)|-9\ge 5$ edges. It follows that $\sum_{X\in\mathcal{X}}\mu(X)\geq2\cdot4-3+(2\cdot2-3)\cdot3+5>2|V(G)|-2$, a contradiction.
\end{proof}

Assume for a contradiction that for some $X_0\in \mathcal{X}$ with $|X_0|\ge 5$, $\mu'(X_0)<0$.   We will simply use $V_i$ and $t_i$ to denote $V_i(X_0)$ and $t_i(X_0)$ for $i=0,1,2$, respectively. Note that $|X_0|=t_0+t_1+t_2$.

\begin{claim}\label{46c2}
$t_1+2t_2\leq5$.
\end{claim}

\begin{proof}
Suppose to the contrary that $t_1+2t_2 \ge 6$.  Then by (R1)-(R4), $$\mu'(X_0)= 2|X_0|-3-2t_0 -\frac{3}{2}t_1 - t_2 =\frac{1}{2}(t_1 +2t_2)-3\ge 0,$$
a contradiction to the assumption that $\mu'(X_0)<0$.
\end{proof}

Note that $V(G)-X_0\neq \emptyset$, for otherwise, if $X_0 = V(G)$, then $\mathcal{X}$ contains at least $V$ and $X_m$, which implies that $\sum_{i=1}^{m} \mu(X_i) \ge (2|V| -3) + (2|X_m|-3) = 2|V|-2$, contradicting (\ref{46mu}).

\begin{claim}\label{46c3}
$V_0\cup V_1$ is not an independent set.
\end{claim}

\begin{proof}
By Claim~\ref{46c2} that $t_1+2t_2\leq5$, we have $t_2\leq2$. Since $|X_0|\geq5$, we have $|V_0\cup V_1| = |X_0| - t_2 \ge 3$. For each $x\in V_0\cup V_1$,  $x$ has at most one neighbor outside of $X_0$ and at most 2 neighbors in $V_2$. Since $d(x)\geq 4$,   $x$ must have a neighbor in $V_0\cup V_1$.
\end{proof}

\begin{claim}\label{46c4}
$V(G)-X_0$ is an independent set.
\end{claim}

\begin{proof}
Suppose to the contrary that $V(G)-X_0$ is not an independent set.

We first claim that $V_0$ is an independent set if $V_0\not=\emptyset$.  For otherwise, $V_1\cup V_2$ is a cut of size at most $t_1+t_2\le 5$ such that both $V_0$ and $G-X_0$ are non-trivial, a contradiction.

We claim that each edge $e\in E(G-X_0)$ must be incident with a vertex in $V'_1$. Otherwise, by Claim \ref{46c3}, $V_0\cup V_1$ is not an independent set, so $V'_1\cup V_2$ is an essentially cut and $|V'_1\cup V_2|\leq t_1+t_2\leq5$ since $t_1+2t_2\leq5$, contrary to the fact that $G$ is essentially 6-connected.


First let $V_0=\emptyset$.  As $5\le |X_0|=t_1+t_2$, and $t_1+2t_2\le 5$ by Claim~\ref{46c2}, $t_2=0$ and $t_1=5$. As each vertex in $V_1$ has degree at least $4$ and has exactly one neighbor not in $X_0$, each vertex in $V_1$ has at least three neighbors in $V_1$. Observe that $|V_1'|\ge 2$, for otherwise, $V_1'$ is a cut of size one. Then some vertex in $V_1'$ has at most $\lfloor\frac{5}{|V_1'|}\rfloor$ neighbors in $V_1$ thus has at least $4-\lfloor\frac{5}{|V_1'|}\rfloor$ neighbors in $G-X_0$. Note that $V_1'$ is a cut if $V(G)-X_0\not=V_1'$. If $|V_1'|\le 3$, then we should have $4-\lfloor\frac{5}{|V_1'|}\rfloor\le |V_1'|-1$, which is impossible.  Therefore, $|V_1'|\ge 4$, and  at least three vertices in $V_1'$ have exactly one neighbor in $V_1$. If there exists $x_1'x_2'\in E(G)$ with $x_1', x_2'\in V_1'$ such that $x_1, x_2$ are the unique neighbors of $x_1', x_2'$ in $V_1$, respectively, then $V_1'\cup \{x_1, x_2\}-\{x_1', x_2'\}$ is an essential cut of size at most five, a contradiction. Then we must have that $x'\in V_1'$ has a unique neighbor $x\in V_1$ such that $x'w\in E(G)$ and $w\in G-X_0-V_1'$, which implies that $V_1'-x'+x$ is an essential cut of size at most five, a contradiction again.

Now let $V_0\not=\emptyset$.  Since $V_0$ is independent, the neighbors of vertices in $V_0$ are in $V_1\cup V_2$.  As each vertex has degree at least four, $t_1+t_2\ge 4$. As $t_1+2t_2\le 5$, we have $t_2=1$ and $t_1=3$, or $t_2=0$ and $4\le t_1\le 5$. In either case, at least three vertices of $V_1$ are adjacent to vertices in $V_0$.  Let $xx'\in E(G)$ with $x\in V_1$ and $x'\in V_1'$ such that $x$ has a neighbor in $V_0$. Then $S=(V_1\cup V_2)-x+x'$ is a cut of size at most five such that the component containing $x$ is non-trivial.  So $G-X_0-x'$ must be trivial, that is, $x'$ must be adjacent to {\em all} edges in $G-X_0$.  Let $|V-X_0|=t$. Then  $$3(t-1)\le |E[X_0, V-X_0-x']|\le t_2\cdot (t-1)+t_1-1.$$
It follows that $t\le 1+\frac{t_1-1}{3-t_2}$.  So $t\le 2$. Since $V(G)-X_0$ is not an independent set, $t\ge 2$.  Therefore $t=2$. Let $V(G)-X_0=\{x', x''\}$.    Then $6\le |E[X_0, \{x',x''\}]|\le 2t_2+t_1\le 5$, by Claim~\ref{46c2}, a contradiction.
\end{proof}

By Claim~\ref{46c4},  assume that $V(G)-X_0$ consists of $t$ isolated vertices. Since $G$ is 4-connected, each of these $t$ vertices has degree at least 4.  Then
$$4t\le |E[X_0,V-X_0]|\le t_1\cdot 1+t_2\cdot t\le (5-2t_2)+t\cdot t_2,$$
which implies that $t\le \frac{5-2t_2}{4-t_2}<2$. So let $V(G)-X_0=\{y\}$.

Since $G$ is 4-connected, $4\le d(y)\le t_1 + t_2$. Then
$$\mu'(X_0)=2|X_0|-3-2t_0-\frac32t_1-t_2=\frac12t_1+t_2-3=\frac12(t_1+2t_2)-3\ge 2-3=-1.$$

For each $X\in \mathcal{X}(y)$ with $|X|>2$, by definition of $V_1$ and $V_0$,  $X\cap (V_0\cup V_1)=\emptyset$ and $X\cap V_2\not=\emptyset$.  So $|X|\le 1+t_2\le 3$.  Therefore, $\mu'(X)\ge 0$ by our discharging rules.

 Therefore, for every $X\in \mathcal{X}-X_0$, $\mu'(X)\ge 0$ and $\mu'(X_0)\ge -1$.  It follows that
$$2|V|-2>\sum_{X\in \mathcal{X}} \mu(X)=\sum_{X\in \mathcal{X}} \mu'(X)+\sum_{x\in V(G)}\mu(x)\ge -1+0+2|V|\ge 2|V|-1,$$
a contradiction, which completes the proof.

\section{The proof of Theorem \ref{th3ess9}}

Suppose to the contrary that the statement is not true. Let $G=(V, E)$ be a counterexample such that $|V(G)|$ is as small as possible, and subject to this condition, $|E(G)|$ is as large as possible.
Thus $G$ is not redundantly rigid, and there exist an edge $e=uv$ such that $T=G-e$ is not rigid.
By Theorem~\ref{th1}, there exists a cover $\mathcal{T}=\{Y_1,Y_2,...,Y_t\}$ of $T$ such that $\sum_{i=1}^t (2|Y_i|-3) < 2|V|-3$.
Without loss of generality, we may assume that $\mathcal{T}=\{Y_1,Y_2,...,Y_t\}$ is a cover of $T$ to minimize $\sum_{i=1}^t (2|Y_i|-3)$.
Let $m=t+1$, $X_i = Y_i$   for $i=1,2, \cdots, t$ and $X_m = X_{t+1} = \{u, v\}$.
Let $\mathcal{X}=\{X_1,X_2,\cdots, X_m\}$, which is then a cover of $G$.  For each $X\in \mathcal{X}$, let $\mu(X)=2|X|-3$. Then
\begin{equation}
\label{39mu}
\sum_{i=1}^{m} \mu(X_i)=\sum_{i=1}^{m}(2|X_i|-3) = \sum_{i=1}^{t}(2|Y_i|-3) + (2|X_m| -3) < 2|V| -2.
\end{equation}

For each $x\in V(G)$, let $\mathcal{X}(x)=\{X_i\in\mathcal{X}: x\in X_i\}$.

\noindent
\textbf{Remark (a):} For any $v_1,v_2\in X_i$, if $|N(v_1)\cup N(v_2)|\geq9$, then $v_1v_2\in E(G)$. Otherwise, the new graph by adding an edge $v_1v_2$ is still 3-connected, essentially $9$-connected and not a redundantly graph, which contradicts the maximality of $|E(T)|$.

\noindent
\textbf{Remark (b):} By the minimality of $\sum_{i=1}^{m-1}(2|X_i|-3)$, we can conclude that each $X_i$ of size at most 3 is a complete graph and if $|X_i|=4$, then $X_i$ induces at least 5 edges. Otherwise, we can replace $X_i$ with a few $X$'s of size 2, each of which induces an edge of $E(X_i)$,  to reduce the value of $\sum_{i=1}^{m-1}(2|X_i|-3)$.

\begin{claim}\label{c1}
In graph $G$, every $3$-vertex is only adjacent to $6^{+}$-vertices.
\end{claim}

\begin{proof}
Here we give a brief idea of the proof.  The proof involves routine case analysis and is tedious, which is included in the appendix. Let $x\in V(G)$ be a $3$-vertex and $N(x)=\{v_1,v_2,v_3\}$. Suppose that $v_1$ is a $5^-$-vertex. Let $S=\{v_2,v_3\}\cup (N(v_1)-x)$.  Then $|S|\le 6$ and one component in $G-S$ contains edge $xv_1$.  As $G$ has no essential cut of size less than $9$, the other components of $G-S$ must form an independent set, say $T=\{u_1, \ldots, u_t\}$.  Note that $N(u_i)\subseteq S$.  We will find a spanning minimally rigid subgraph, namely a subgraph $H$ such that $|E(H)|=2|V(H)|-3$ and every subgraph $H'$ of $H$ satisfies $|E(H')|\le 2|V(H')|-3$, in $G-e$ for every $e\in E(G)$, which implies that $G$ is not a counterexample.
\end{proof}

In the rest of the proof, we will use a discharging argument to  reach a contradiction.   We will distribute the $\mu(X)$ value of each $X\in \mathcal{X}$ to vertices of $X$, and we denote by $\sigma_x(X)$ the value that $x$ gets from $X$. The following are the discharging rules.

\begin{enumerate}[(R1)]
\item Suppose $|X|=2$. Then $\mu(X)=1$.

  \begin{enumerate}
   \item [(R1A)] If $X=\{v_1, v_2\}$ with $d(v_1)=3$, then $\sigma_{v_1}(X)=\frac23$ and $\sigma_{v_2}(X)=\frac13$.
   \item[(R1B)] If $X$ contains two $4^+$-vertices, then each vertex gets $\frac12$ from $X$.
  \end{enumerate}
 \smallskip
\item Suppose $|X|=3$. Then $\mu(X)=3$.
  \begin{enumerate}
   \item [(R2A)] If $X=\{v_1, v_2, v_3\}$ with $d(v_1)=3$, then $\sigma_{v_1}(X)=\frac43$ and $\sigma_{v_2}(X)=\sigma_{v_3}(X)=\frac56$.
   \item[(R2B)] If $X$ contains three $4^+$-vertices, then each vertex gets $1$ from $X$.
  \end{enumerate}
  \smallskip
\item Suppose $|X|=4$. Then $\mu(X)=5$.
\begin{enumerate}
   \item [(R3A)] If $X$ contains a 3-vertex $v_1$ and $N(v_1)\subseteq X$, then $\sigma_{v_1}(X)=2$ and each of other vertices gets 1 from $X$.
   \item [(R3B)] If $X$ contains a $4^-$-vertex $v_1$ such that $\mathcal{X}(v_1)=\{X_1,X_2\}$ with $|X_1|=2,|X_2|=4$, then $\sigma_{v_1}(X)=\frac32$ and each of other vertices gets 1 from $X$.
   \item [(R3C)] Otherwise, each vertex of $X$ gets $\frac54$ from $X$.
  \end{enumerate}
  \smallskip
\item Suppose $|X|\geq5$. Then $\mu(X)=2|X|-3$.
  \begin{enumerate}
   \item [(R4A)] If $(N(x)\cup \{x\})\subseteq X$, then $\sigma_{x}(X)=2$.
   Let $V_0 (X)$ be the set of all those vertices $x$ in $X$, and denote $|V_0(X)|=t_0(X)$.
 \smallskip

\item[(R4B)] Suppose $x\in X$ has only one adjacent vertex $x'\notin X$.
 \begin{enumerate}
  \item[(R4B1)] If $d(x')=3$, and there is an $X'=\{x,x'\}\in \mathcal{X}$, then $\sigma_{x}(X)=\frac53$.
  Let $V_1 (X)$ be the set of all those vertices $x$ in $X$, and let $|V_1(X)|=t_1(X)$ and $V'_1(X)=\{x': x'\in N(x)$ and $x'\notin X, x\in V_1(X)\}$.
   \item[(R4B2)] If $x\not\in V_1$, then $\sigma_{x}(X)=\frac32$.
   Let $V_2 (X)$ be the set of all those vertices $x$ in $X$, and let $|V_2(X)|=t_2(X)$ and $V'_2(X)=\{x': x'\in N(x)$ and $x'\notin X, x\in V_2(X)\}$.
   \end{enumerate}
 \smallskip
\item[(R4C)] Suppose $x\in X$ has at least two adjacent vertices outside of $X$.
  \begin{enumerate}
   \item[(R4C1)] Suppose $x$ has exactly two neighbors $x'$ and $x''$ outside of $X$.
    If $d(x')=d(x'')=3$ and there are $X_1=\{x,x'\},X_2=\{x,x''\}\in \mathcal{X}$, then $\sigma_{x}(X)=\frac43$. Let $V_3 (X)$ be the set of all those vertices $x$ in $X$, and let $|V_3(X)|=t_3(X)$ and $V'_3(X)=\{x': x'\in N(x)$ and $x'\notin X, x\in V_3(X)\}$.
  \item[(R4C2)] Suppose $x$ has exactly two neighbors $x'$ and $x''$ outside of $X$. 
  If $d(x')=3$ and $x\not\in V_3$, then $\sigma_{x}(X)=\frac76$. Let $V_4 (X)$ be the set of all those vertices $x$ in $X$, and let $|V_4(X)|=t_4(X)$ and $V'_4(X)=\{x': x'\in N(x)$ and $x'\notin X, x\in V_4(X)\}$.
   \item[(R4C3)] Otherwise, $\sigma_{x}(X)=1$. Let $V_5 (X)$ be the set of all those vertices $x$ in $X$, and let $|V_5(X)|=t_5(X)$ and $V'_5(X)=\{x': x'\in N(x)$ and $x'\notin X, x\in V_5(X)\}$.
  \end{enumerate}
  \end{enumerate}
\end{enumerate}


\vspace{0.1in}
For $x\in V(G)$, let $\mu(x)=\sum_{X\in \mathcal{X}(x)}\sigma_x(X)$, i.e., the total charge that $x$ gets from $\mathcal{X}(x)$.

\begin{claim}\label{c3}
$\mu(x)\geq2$, for every $x\in V(G)$.
\end{claim}

\begin{proof}
First let $d(x)=3$.  By (R3A), we may assume that $|\mathcal{X}(x)|\ge 2$.  If $|\mathcal{X}(x)|=3$, then $x$ is in three sets of size at least two, thus by (R1A), $\mu(x)=\frac23\cdot 3=2$. If $|\mathcal{X}(x)|=2$, then $x$ is in a set of size at least two and and in a set of size at least three, thus by (R1A), (R2A), (R3) and (R4), $\mu(x)\ge \frac23+\frac43=2$.

Let $d(x)\ge 4$. By (R4A), we may assume that $|\mathcal{X}(x)|\ge 2$. We first assume that $x$ has no neighbors of degree $3$.  If each $X\in \mathcal{X}(x)$ has $|X|\ge 3$, then by (R2)-(R4), $x$ gets at least $1$ from each $X\in \mathcal{X}(x)$, thus $\mu(x)\ge 2$. So we may assume that $|\mathcal{X}(x)|\ge 2$ and $t$ sets of $\mathcal{X}(x)$ have size two.  By (R1B) and (R2)-(R4), if $t=d(x)$, then $\mu(x)\ge \frac12 t\ge 2$; if $t<d(x)$, then $x$ gets at least $1$ from sets of size more than two, thus $\mu(x)<2$ only if $t=1$ and $|\mathcal{X}(x)|=2$, in which case, by (R3B) and (R4B2), $\mu(x)\ge \frac12+\frac32=2$.

Now we assume that $x$ has a neighbor of degree $3$. By Claim~\ref{c1}, $d(x)\ge 6$.  If $\mathcal{X}(x)$ does not contain sets of size two, then by (R2)-(R4), $x$ gets at least $\frac56$ from each $X\in \mathcal{X}(x)$, so $\mu(x)<2$ only if $x$ gets a $\frac56$ by (R2A) and a $1$ by (R4C3), but these two cases cannot happen at the same time since $d(x)\ge 6$.  So assume that $t\ge 1$ sets of $\mathcal{X}(x)$ have size two.        If $|\mathcal{X}(x)|=2$, then $t=1$ and by (R1) and (R4B), $\mu(x)\ge\min\{\frac13+\frac53,\frac12+\frac32\}=2$. So let $|\mathcal{X}(x)|\ge 3$.  Assume that $\mathcal{X}(x)$ contains $s$ sets of size three. If $t+2s=d(x)$, then by (R1A) and (R2A), $\mu(x)=\frac13 t+\frac56 s\ge \frac13(t+2s)\ge 2$. So let $t+2s<d(x)$, then by (R1A) and (R2)-(R4), $\mu(x)\ge \frac13 t+\frac56s+1$. Clearly, as $t\ge 1$, $\mu(x)\ge 2$ if $s\ge 1$. So let $s=0$.  As $|\mathcal{X}(x)|\ge 3$ and by (R2)-(R4), $x$ gets either $\frac13$ or at least $1$ from each set in $\mathcal{X}(x)$, we have $\mu(x)<2$ only if $t=2$ and $|\mathcal{X}(x)|=3$, in which case, by (R4C1), $\mu(x)\ge \frac13\cdot 2+\frac43=2$.
\end{proof}

Now we prove that for each $X\in \mathcal{X}$,  the final charge $\mu'(X)\ge 0$.

\begin{claim}\label{order-4}
For each $X\in \mathcal{X}$ with $|X|\le 4$, $\mu'(X)\ge 0$.
\end{claim}

\begin{proof}
If $|X|=2$, then $\mu'(X)=2\cdot2-3-\max\{\frac23+\frac13,\frac12\cdot2\}=0$ by Claim \ref{c1} and (R1).
If $|X|=3$, then $\mu'(X)=2\cdot3-3-\max\{\frac43+\frac56\cdot2,1\cdot3\}=0$ by Claim \ref{c1} and (R2). If $|X|=4$, then $\mu(X)=5$. We know that there are three $6^+$-vertices and one 3-vertex in (R3A) by Claim \ref{c1}. We will show that $X$ contains at most two 4-vertices when $\mathcal{X}(x)=\{X_1,X_2\}$ with $|X_1|=2,|X_2|=4$ for some $x\in X$. Thus, $\mu'(X)\geq 5-\max\{2+1\times3,\frac32\cdot2+1\cdot2,\frac54\cdot4\}=0$.

Suppose to the contrary that $X=\{v_i:  i\in[4]\}$ in which $|X_1(v_i)|=2$ and $|X_2(v_i)|=4$ for $i\in [3]$. Let $X_1(v_i)=\{u_i, v_i\}$.  Since for $i\in [3]$, $d(v_i)=4$ and $v_i$ has exactly one neighbor outside of $X$, $X$ induces a clique. If $V(G)-X-\{u_1,u_2,u_3\}$ induces at least one edge $e'$, then $S=\{u_1,u_2,u_3,v_4\}$ separates $\{v_1,v_2,v_3\}$ and $e'$. We obtain an essentially 4-cut $S$, a contradiction. Thus we may assume that $V(G)-X-\{u_1,u_2,u_3\}$ is an independent set. As $|V(G)|\geq 10$, $V(G)-X-\{u_1,u_2,u_3\}\neq \emptyset$. Since $G$ is 3-connected, every vertex in $V(G)-X-\{u_1,u_2,u_3\}$ is adjacent to at least three vertices in $S$. Without loss of generality, we can assume that $w\in V(G)-X-\{u_1,u_2,u_3\}$ is adjacent to $u_1$. Now we choose $S'=\{v_1,v_4,u_2,u_3\}$ separating $\{v_2,v_3\}$ and $V(G)-X-\{u_2,u_3\}$. Note that $v_2v_3,wu_1\in E(G)$, so we get that $S'$ is an essentially 4-cut, a contradiction.
\end{proof}

So we may only consider $X\in \mathcal{X}$ with $|X|\ge 5$.  Assume for a contradiction that for some $X_0\in \mathcal{X}$ with $|X_0|\ge 5$, $\mu'(X_0)<0$.  We will simply use $V_i$, $V_i'$, and $t_i$ to denote $V_i(X_0)$, $V_i'(X_0)$ and $t_i(X_0)$ for $i\in [5]\cup \{0\}$. Note that $|X_0|=\sum_{i=0}^5 t_i$.

\begin{claim}\label{c4}
$2t_1+3t_2+4t_3+5t_4+6t_5\leq17$.
\end{claim}

\begin{proof}
By (R4),
$0>\mu'(X_0)=(2|X_0|-3)-2t_0- \frac{5}{3}t_1 - \frac{3}{2}t_2 -\frac{4}{3}t_3- \frac76t_4- t_5= \frac13t_1+\frac12t_2+\frac23t_3+\frac56t_4+t_5-3.$
It follows that $2t_1+3t_2+4t_3+5t_4+6t_5<18$, so the inequality in the claim.
\end{proof}

Notice that $V(G)-X_0\neq \emptyset$, for if $X_0 = V(G)$, then $\{V, X_m\}\subseteq\mathcal{X}$, which implies that $\sum_{X\in\mathcal{X}} \mu(X) \ge (2|V| -3) + (2|X_m|-3)= 2|V|-2$, contradicting (\ref{39mu}).

\begin{claim}\label{c6}
$V(G)-X_0$ is an independent set.

\end{claim}

\begin{proof}
Assume on the contrary that $V(G)-X_0$ is not an independent set. We will reach a contradiction by a series of claims. 

\medskip
(c0) $V_0$ is independent if $V_0\not=\emptyset$.

\begin{proof}[Proof of (c0)]
For otherwise, $X_0-V_0$ is an essential cut of order at most $8$, by Claim~\ref{c4}.
\end{proof}

(c1) $e(V_0, V_1\cup V_3)=0$.

\begin{proof}[Proof of (c1)]
Let $e=uv\in E(G)$ with $u\in V_0$ and $v\in V_1\cup V_3$.

Assume that $v\in V_1$ and $vw\in E(G)$ with $w\in V_1'$.  By definition $d(w)=3$ and by Claim~\ref{c1}, $d(v)\ge 6$.  Since $X_0-V_0-v+w$ is a cut of order $\sum_{i=1}^5 t_i\le 8$ and is not an essential cut, all edges in $E(G-X_0)$ must be incident with $w$. Let $ww'\in E(G)$ with $w'\not\in X_0$. By Claim~\ref{c1}, $d(w')\ge 6$.  Since $ww'$ is the only edge incident with $w'$ that is in $E(G-X_0)$, all other neighbors of $w'$ are in $V_2\cup V_4\cup V_5$. In particular, $t_2+t_4+t_5\ge 5$.  By Claim~\ref{c4}, $t_1=1, t_2=5$ and $t_3=t_4=t_5=0$, and the only neighbor of vertices of $V_2$ in $V(G)-X_0$ is $w'$.  Therefore, $w$ has no neighbor other than $v$ and $w'$, a contradiction to $d(w)=3$.

Let $v\in V_3$ and $vw_1, vw_2\in E(G)$ with $w_1, w_2\in V_3'$.  By definition $d(w_1)=d(w_2)=3$, and by Claim~\ref{c1}, $d(v)\ge 6$.  Note that $X_0\cup \{w_1,w_2\}-V_0-v$ is a cut, whose order is at most $(t_1+t_2+t_3+t_4+t_5)+1\le \frac{1}{2}(17-2t_3)+1<9$,  and is not an essential cut, all edges in $E(G-X_0)$ must be incident with $w_1$ or $w_2$. Let $w_1w'\in E(G)$ with $w'\not\in X_0$.  By Claim~\ref{c1}, $d(w')\ge 6$.  Since $w'w_1, w'w_2$ are the only potential edges incident with $w'$ that are in $E(G-X_0)$, all other neighbors of $w'$ are in $V_2\cup V_4\cup V_5$. In particular, $t_2+t_4+t_5\ge 4$.  By Claim~\ref{c4}, $t_3=1, t_2=4$ and $t_1=t_4=t_5=0$, and the only neighbor of vertices of $V_2$  in $V(G)-X_0$ is $w'$.  Therefore, $w_1, w_2$ have no neighbor other than $v$ and $w'$, a contradiction to $d(w_1)=d(w_2)=3$.
\end{proof}



(c2) $t_1=0$.

\begin{proof}[Proof of (c2)] Let $v_1\in V_1$. Then $d(v_1)\ge 6$ and $v_1$ is adjacent to a $3$-vertex $u_1\in V'_1$.

We claim that $v_1$ has no neighbors in $V_1$. For otherwise, let $v_1v_1'\in E(G)$ with $v_1'\in V_1$. Let $v_1'u_1'\in E(G)$ with $u_1'\in V_1'$. By definition, $d(u_1)=d(u_1')=3$. Then $X_0-V_0-\{v_1, v_1'\}+\{u_1, u_1'\}$ is a cut of order at most $\sum_{i=1}^5 t_i\le 8$ and is not an essential cut. It follows that all edges in $E(G-X_0)$ are incident with $u_1$ or $u_1'$. Let $u_1u\in E(G)$ with $u\not\in X_0$. By Claim~\ref{c1}, $d(u)\ge 6$, and $N(u)\subseteq \{u_1, u_1'\}\cup V_2\cup V_4\cup V_5$. It follows that $t_2+t_4+t_5\ge 4$. By Claim~\ref{c4}, $t_1=2, t_2=4$ and $t_3=t_4=t_5=0$. In particular, the only neighbor of vertices of $V_2$ in $V(G)-X_0$ is $u$.  Then $u_1, u_1'$ have no neighbors other than $v_1, v_1'$ and $u$.  So $u_1=u_1'$, but then $u$ should have $5$ neighbors in $V_2$, a contradiction again.

From above and (c1), $N(v_1)\subseteq X_0-V_0-V_1$. As $d(v_1)\ge6$, $t_2+t_3+t_4+t_5\ge 5$. By Claim \ref{c4}, we have that $t_1=1,t_2=5,t_3=t_4=t_5=0$ and $v_1$ is adjacent to every vertex in $V_2$. Let $N(u_1)=\{v_1,w_1,w_2\}$.  By definition and the minimality of $\sum_{i=1}^{m-1}(2|X_i|-3)$, $w_1,w_2\not\in V_2$ thus $w_1, w_2\in V(G)-X_0$. By Claim \ref{c1}, $d(w_i)\ge6$ for $i\in[2]$ . It follows that $w_1$ or $w_2$ (say $w_1$) is adjacent to at most two vertices in $V_2$. Let $v_2\in V_2$ be adjacent to $u_2\in V_2'$ and $v_2w_1\notin E(G)$. As $d(w_1)\ge6$, $w_1$ must be adjacent to a vertex in $V(G)-X_0-\{u_1,w_1,u_2\}$. So $V_2\cup\{u_1,u_2\}-v_2$ is an essential $6$-cut, a contradiction.
\end{proof}

(c3) $t_3=0$.

\begin{proof}[Proof of (c3)]
For otherwise, let $v_1\in V_3$ such that $u_1v_1,u_2v_1\in E(G)$ for 3-vertices $u_1,u_2\in V'_3$.  By Claim~\ref{c1}, $d(v_1)\ge 6$. By (c1) and (c2), $4\le e(v_1, X_0)\le t_2+(t_3-1)+t_4+t_5$. From Claim~\ref{c4}, $t_4=t_5=0$, and $3t_2+4t_3\le 17$. So $(t_2, t_3)\in \{(3,2), (4,1)\}$ and $v_1$ is adjacent to all vertices in $V_2\cup V_3-v_1$.

By definition and the minimality of $\sum_{i=1}^{m-1}(2|X_i|-3)$, $u_1$ and $u_2$ have no neighbor in $V_2$. Let $W=(N(u_1)\cup N(u_2))\cap(V(G)-X_0)$. If $W=\{w_1\}$, then $t_2=3,t_3=2$ and the neighbors of vertices of $V_3$ are $u_1$ and $u_2$. Since $d(w_1)\ge6$, $w_1$ has a neighbor in $V(G)-X_0-\{u_1,u_2\}$. Then $V_2\cup\{u_1,u_2\}$ is an essential 5-cut, a contradiction. So we may assume that $\{w_1,w_2\}\subseteq W$. As $t_2\le4$, some vertex in $W$, say $w_1$,  is adjacent to at most two vertices in $V_2$. Let $v_2\in V_2$ be not adjacent to $w_1$ and $v_2u'\in E(G)$ for $u'\in V'_2$. As $d(w_1)\ge6$, $w_1$ must have a neighbor in $V(G)-X_0-\{u_1,u_2,u'\}$. Therefore, $V_2\cup V_3\cup \{u_1,u_2,u'\}-\{v_1,v_2\}$ is an essential 6-cut, a contradiction.
\end{proof}

(c4) $t_5<2$.

\begin{proof}[Proof of (c4)]
Assume that $t_5\ge 2$. By Claim~\ref{c4}, $t_5=2$ and $t_2+t_4\le 1$.  As $|X_0|\ge 5$, $t_0\ge 2$ and by (c0), $t_2+t_4=1$ and each vertex in $V_0$ is adjacent to all three vertices in $V_2\cup V_4\cup V_5$.    Let $V_0=\{v_1,v_2,\dots\}, V_2\cup V_4=\{u_1\}, V_5=\{u_2,u_3\}$.   Now $\{u_2, u_3\}\cup (N(u_1)-V_0)$ is a $4^-$-cut and the component $V_0\cup \{u_1\}$ contains edges, $V(G)-X_0-N(u_1)$ must be an independent set.  When $u_1\in V_4$, let $V_4'=N(u_1)\cap (V(G)-X_0)=\{u, u'\}$ with $d(u')=3$.  Note that $N(u')\subseteq V_4\cup V_5\cup u$, for otherwise, the vertex $w\in N(w_1)-X_0-u$ has degree at least $6$, thus $w$ must have a neighbor in $V(G)-X_0-N(u_1)$. It follows that $\{u,w_1, u_2, u_3\}$ is an essential cut of size $4$, a contradiction.  When $u_1\in V_2$, we let $V_2'=\{u\}$.  Since $\{u, u_2, u_3\}$ is a cut and the component containing $V_0$ contains edges, all edges in $V(G)-X_0$ must be adjacent to $u$.   Let $W=V(G)-X_0-N(u_1)=\{w_1,w_2,\dots, w_k\}$ such that $N(w_i)=\{u, u_2, u_3\}$ for $i\ge 1$. We will find a spanning minimally rigid subgraph $H$ in $G-e$ for every $e\in E(G)$.

\begin{itemize}
\item $t_0=2$. We claim that there exists either a $K_{3,3}$ or a $K_3$ in $G-e$. Note that the subgraph induced by $\{u, u_2, u_3\}\cup W$ contains a $K_{3, k}$.
    If $u_1\in V_2$, then $k\ge4$ since $|V(G)|\ge 10$. So there exists a $K_{3,3}$ in $G-e$. If $u_1\in V_4$, then $k\ge 3$. So there exists a $K_{3,3}$ in $G-e$ unless $k=3$ and $e\in E[\{u, u_2, u_3\},W]$. Next, we consider the case that $k=3$ and $e\in E[\{u, u_2, u_3\},W]$. If $N(u')\subseteq V_4\cup V_5$, then $u'$ is adjacent to $u_1,u_2$ and $u_3$. So the subgraph induced by $\{u_1, u_2, u_3\}\cup \{v_1, v_2,u'\}$ contains a $K_{3,3}$. If $u'$ is adjacent to $u$, then we assume that $N(u')=\{u_1,u_2,u\}$. In $G-e$, there exists a triangle $u_1u'u$.

If there exists a $K_{3,3}$ in $G-e$, then we choose a $K_{3,3}$ from $G-e$. Then we can easily add each of the rest vertices one by one by joining it with two vertices in the already-formed subgraph to obtain a spanning subgraph $H$ in $G-e$.  It is clear that $e(H)=3\cdot3+2(n-6)=2n-3$,  and for any subgraph $H'$ of $H$, let $H^0=H'\cap K_{3,3}$, then
$$e(H')\le2n(H'-H^0)+e(H^0)\le 2n(H'-H^0)+2n(H^0)-3=2n-3.$$

If there exists a $K_3$ in $G-e$, then we choose a $K_3$ from $G-e$. Then add $u_2$ and then easily add each of the rest vertices one by one by joining it with two vertices in the already-formed subgraph to obtain a spanning subgraph $H$ in $G-e$, since $u_1$ ia adjacent to all vertices of $X_0$. It is clear that $e(H)=3+2(n-3)=2n-3$,  and for any subgraph $H'$ of $H$, let $H^0=H'\cap K_{3}$, then
$$e(H')\le2n(H'-H^0)+e(H^0)\le 2n(H'-H^0)+2n(H^0)-3=2n-3.$$
\item $t_0\ge 3$. Note that every vertex in $V_0$ is adjacent to $u_1,u_2$ and $u_3$ in $G$. So there exists a $K_{3,3}$ in $G$. Now we claim that there exists either a $K_{3,3}$ or a $K_{2,3}+e'$ in $G-e$. Clearly, there exists a $K_{3,3}$ in $G-e$ unless $t_0=3$ and $e\in E[\{u_1, u_2, u_3\}, V_0]$. Next, we consider the case that $t_0=3$ and $e\in E[\{u_1, u_2, u_3\}, V_0]$. If $k\ge3$, then the subgraph induced by $\{u, u_2, u_3\}\cup \{w_1, \ldots, w_k\}$ contains a $K_{3,3}$ in $G-e$. If $k\le2$, then $u_1\in V_4$ and $k=2$ since $t_0=3$ and $|V(G)|\ge 10$. In this case, if $N(u')\subseteq V_4\cup V_5$, then $u'$ is adjacent to $u_1,u_2$ and $u_3$. So the subgraph induced by $\{u_1, u_2, u_3\}\cup \{v_1, v_2,v_3,u'\}$ contains a $K_{3,3}$. If $u'$ is adjacent to $u$, then $u$ is adjacent to $u_2$ since $u$ is a $6^+$-vertex. So the subgraph induced by $\{u, u_2, u_3\}\cup \{u_1, u_2\}$ contains a $K_{2,3}+e'$.


If there exists a $K_{3,3}$ in $G-e$, then we choose a $K_{3,3}$ from $G-e$. Then add each of the rest vertices one by one by joining it with two vertices in the already-formed subgraph to obtain a spanning subgraph $H$. Again,  $e(H)=3\cdot3+2(n-6)=2n-3$, and for any subgraph $H'$ of $H$, let $H^0=H'\cap K_{3,3}$, then
$$e(H')\le2n(H'-H^0)+e(H^0)\le 2n(H'-H^0)+2n(H^0)-3=2n-3.$$

 If there exists a $K_{2,3}+e'$ in $G-e$, then we choose a $K_{2,3}+e'$ from $G-e$. Then add each of the rest vertices one by one by joining it with two vertices in the already-formed subgraph to obtain a spanning subgraph $H$. Again,  $e(H)=3\cdot2+1+2(n-5)=2n-3$, and for any subgraph $H'$ of $H$, let $H^0=H'\cap K_{2,3}+e'$, then
$$e(H')\le2n(H'-H^0)+e(H^0)\le 2n(H'-H^0)+2n(H^0)-3=2n-3.$$
\end{itemize}
Therefore $G$ contains a spanning minimally rigid subgraph in all cases, a contradiction to the assumption that $G$ is a counterexample.
\end{proof}

(c5) $t_4=0$.

\begin{proof}[Proof of (c5)]
For otherwise, let $v_1\in V_4$ such that $u_1v_1,u_2v_1\in E(G)$ for vertices $u_1,u_2\in V'_4$ such that $d(u_1)=3$.

Let $N(u_1)=\{v_1, w_1, w_2\}$ such that $w_1\not=u_2$. By Claim~\ref{c1}, $d(v_1), d(w_1),d(w_2)\ge 6$.  Assume that $w_1\in X_0$.  Then $w_1\in V_4\cup V_5$ and thus $t_4+t_5\ge2$. By Claim \ref{c4}, $|X_0-V_0|=t_2+t_4+t_5\le4$. It follows that $v_1$ has a neighbor $v$ in $V_0$. Note that $(X_0-V_0-v_1)\cup\{u_1,u_2\}$ is a cut of order at most 5 and is not an essential cut. If $w_2\in V(G)-X_0-u_2$, then all but $w_2u_1$ edges in $E(G-X_0)$ must be incident with $u_2$. If $w_2\notin V(G)-X_0-u_2$, then all edges in $E(G-X_0)$ must be incident with $u_2$. Let $|V(G)-X_0|=t$. Then $t\ge 2$ and $e(X_0, V(G)-X_0)\ge\min\{2(t-3)+(6-2)+[d(u_2)-(t-2)]+2,2(t-1)+[d(u_2)-(t-1)]\}\ge 6+(t-1)=t+5$. On the other hand, $e(X_0, V(G)-X_0)\le t_2+2t_4+t\cdot t_5$.  It follows that $5+t(1-t_5)\le t_2+2t_4\le \frac{17-6t_5-0.5t_2}{2.5}$.  So $2.5t(1-t_5)+6t_5+0.5t_2\le 4.5$, which is impossible as $t\ge 2$ and $t_5\le 1$ by (c4).

So we may assume that $w_1\in V(G)-X_0$. As $t_4\ge1$, we have that $|X_0-V_0|=t_2+t_4+t_5\le5$ by Claim \ref{c4}.  Assume that $v\in V_0$ is a neighbor of $v_1$. Then $C=(X_0-V_0-v_1)\cup \{u_1,u_2\}$ is a cut of order at most 6 and is not an essential cut. Therefore, $w_1$ has no neighbor in $V(G)-X_0-\{u_1,u_2\}$. As $d(w_1)\ge6$, it implies that $t_2=4,t_4=1,t_5=0$ and $w_1$ is adjacent to $\{u_1,u_2\}\cup V_2$. It follows that $w_2$ (which may be $u_2$) has at least three neighbors in $V(G)-X_0-\{u_1,u_2,w_1\}$ from $d(w_2)\ge6$. So there must exist an edge in $E(G-X_0-\{u_1,u_2\})$ and $C$ is an essential cut, a contradiction. Thus, $N(v_1)\cap X_0\subseteq X_0-V_0$. As $d(v_1)\ge6$, $|X_0-V_0|=5$ with $t_2=4,t_4=1,t_5=0$ by Claim \ref{c4} and $v_1$ is adjacent to all vertices in $V_2$. Suppose that $w_1$ is not adjacent to at least two vertices in $V_2$. Let $v_2\in V_2$ be not adjacent to $w_1$ and $v_2u'\in E(G)$ for $u'\in V'_2$. Then $w_1$ has a neighbor in $V(G)-X_0-\{u_1,u_2,u'\}$. Thus, $(X_0-V_0-\{v_1,v_2\})\cup\{u_1,u_2, u'\}$ is an essential cut of order at most 6, a contradiction. So $w_1$ is not adjacent to at most one vertex in $V_2$. As $d(w_2)\ge6$, $w_2$ has at least two neighbors $z_1,z_2$ in $V(G)-X_0-\{u_1,w_2,w_1\}$ and there must be at least one vertex in $\{z_1,z_2\}$, say $z_1$, such that $z_1$ has no neighbor in $V_2$. Let $v\in V_2$ be adjacent to $w_1$. Then $z_1$ has a neighbor in $V(G)-X_0-\{u_1,w_2,w_1\}$ and $(X_0-V_0-\{v_1,v\})\cup\{u_1,w_2,w_1\}$ is an essential cut of order at most 6, a contradiction.
\end{proof}

From (c1)-(c5) and Claim~\ref{c4}, $3t_2+6t_5\le 17$. So $t_2+2t_5\le 5$.

Assume that $t_0=0$. Then $t_2=5$ and $t_5=0$, since $|X_0|\ge 5$. As each vertex in $V_2$ has exactly one neighbor in $V(G)-X_0$, each vertex in $V_2$ has degree at most $5$, in particular by Claim~\ref{c1}, no vertex has degree three.   Assume first that $|V_2'|\le 2$. Then $V(G)=X_0\cup V_2'$, for otherwise, $V_2'$ is a cut of order at most two.  As $V(G)-X_0$ is not independent, $V_2'$ contains two vertices that are adjacent. Then one of them must have degree $3$, but then the other must have degree at least $6$, which is impossible. So $|V_2'|\ge 3$.  Then $V_2'$ contains vertices with exactly one neighbor in $V_2$, say $x_1', \ldots, x_s'$ for some $s\ge 1$. For each $i\in [s]$, let $x_ix_i'\in E(G)$ with $x_i\in V_2$. For each $x_i'$, if $x_i'$ is adjacent to a vertex in $V(G)-X_0-V_2'$, then $V_2'-x_i'+x_i$ is an essential $5$-cut, a contradiction. Thus all neighbors of $x_i'$ are in $V_2\cup V_2'$, in particular $V(G)=X_0\cup V_2'$.  Since $d(x_i)\le 5$, $d(x_i')\ge 4$, which implies that $|V_2'|\ge 4$ and thus $s\ge 4$. We may assume that $x_1'x_2'\in E(G)$, as $d(x_i')\ge 4$.  Then $(V_2'-\{x_1,x_2'\})\cup \{x_1, x_2\}$ is an essential cut of order at most $5$, since $x_1'x_2'\in E(G)$ and $x_3$ has a neighbor in $V_2-\{x_1, x_2\}$, a contradiction.

Now let $V_0\not=\emptyset$.  Then $t_2+t_5\geq3$ since the neighbors of $V_0$ are in $V_2\cup V_5$. By Claim~\ref{c4} and (c4), $3\le t_2\le 5, t_5=0$, or $2\le t_2\le3, t_5=1$. Let $v\in V_2$ with neighbor $v'\in V_2'$ such that $v$ has a neighbor in $V_0$. Then  $C=V_2\cup V_5-v+v'$ is a set of size at most $5$ such that the component containing $v$ contains an edge. Since $G$ is essentially 9-connected, the other components must be trivial and all edges in $V(G)-X_0$ must be adjacent to $v'$. We assume $|V(G)-X_0|=t$. So we have that $$\sum_{v\in V(G)-X_0}d(v)\le t_2+2(t-1)+tt_5=(2+t_5)t+t_2-2$$ for $t\ge2$. By Claim \ref{c1}, $\sum_{v\in V(G)-X_0}d(v)\ge \min\{3t+3,4t\}$. A simple computation gives $t_2=1,t_5=2$, a contradiction to (c4).
\end{proof}

By Claim~\ref{c6}, assume that $V(G)-X_0$ consists of $t\ge 1$ isolated vertices.
Since $G$ is 3-connected, each of these $t$ vertices has degree at least 3, and thus $|E[G-X_0, V_1\cup V_2\cup V_3\cup V_4]|\ge (3-t_5)t$. On the other hand, by definitions of $V_i$ for $1\le i\le 5$ and Claim~\ref{c4}, we have
$$(3-t_5)t\le |E[G-X_0, V_1\cup V_2\cup V_3\cup V_4]|\le t_1 + t_2 + 2 t_3+2t_4\le t_1 + \frac32 t_2 + 2 t_3+2t_4\le \frac12(17-6t_5).$$
Therefore,  $t\le \frac{17-6t_5}{6-2t_5}<3$ and  $V(G)-X_0$ contains at most two isolated vertices.  It follows that $3t\le |E[X_0, G-X_0]|\le t_1 + t_2 + 2 t_3+2t_4+2t_5$. Therefore, $2t_1+3t_2+4t_3+5t_4+6t_5\ge 6t$, and $$\mu'(X_0)=(2|X_0|-3)-2t_0- \frac{5}{3}t_1 - \frac{3}{2}t_2 -\frac{4}{3}t_3-\frac76t_4- t_5=-3+\frac{1}{6}(2t_1 +3t_2 + 4t_3+5t_4 + 6t_5)\ge -3+t\ge -2.$$

We may claim that for each $y\in G-X_0$ and $X\in \mathcal{X}(y)$, $|X|\le 4$. Suppose otherwise that $X\in \mathcal{X}(y)$ with $|X|\ge 5$.  Note that $X$ only needs to cover all edges in $E[X\cap X_0, X\cap (G-X_0)]$.  On the other hand, each edge in $E[X\cap X_0, X\cap (G-X_0)]$ may be covered with a set of order two, that is, edges in $E[X\cap X_0, X\cap (G-X_0)]$ may be covered with at most $e(X\cap X_0, X\cap (G-X_0))$ sets of order two. It follows from the minimality of $\sum_{i=1}^{m-1}(2|X_i|-3)$ that $e(X\cap X_0, X\cap (G-X_0))\ge 2|X|-3\ge 7$.  As $|X|\ge 5$, $X\cap (V_1\cup V_3)=\emptyset$.  Also, $X\cap (G-X_0)$ contains at most two vertices.  Therefore, $$7\le e(X\cap X_0, X\cap (G-X_0))\le t_2+2t_4+2t_5.$$  From Claim~\ref{c4}, $3t_2+5t_4+6t_5\le 17$, so $t_4+t_5\le 3$.  Furthermore, if $t_4+t_5=3$ then $t_2=0$, if $t_4+t_5=2$ then $t_2\le 2$,  if $t_4+t_5= 1$ then $t_2\le 4$, and if $t_4+t_5=0$ then $t_2\le 5$. In any case, $t_2+2(t_4+t_5)\le 6$, a contradiction.



Therefore, $\mu'(X_0)\ge -2$ and by Claim~\ref{order-4} for every $X\in \mathcal{X}-X_0$, $\mu'(X)\ge 0$.  It follows that
$$2|V|-2>\sum_{X\in \mathcal{X}} \mu(X)=\sum_{X\in \mathcal{X}} \mu'(X)+\sum_{x\in V(G)}\mu(x)\ge -2+0+2|V|\ge 2|V|-2,$$
a contradiction, which completes the proof.

\section{Non-rigid graphs that are $3$-connected and essentially $t$-connected for $t\le 8$.}

In this section, we will construct $3$-connected essentially $t$-connected non-rigid graphs for each $t\le 8$.  We start with a $3$-connected bipartite $T=(X,Y)$ such that each $x\in X$ is a $3$-vertex and each vertex $y\in Y$ is a  $t$-vertex. Then $G$ is obtained from $T$ by replacing each $y\in Y$ with $K_t$ and making vertices in $N_T(y)$ adjacent to different vertices of the $K_t$.   See Figure~\ref{fig1} for the $3$-connected essentially $t$-connected for $t\in \{6,7\}$ graphs that are not rigid.

\begin{figure}[ht]
\includegraphics[scale=0.45]{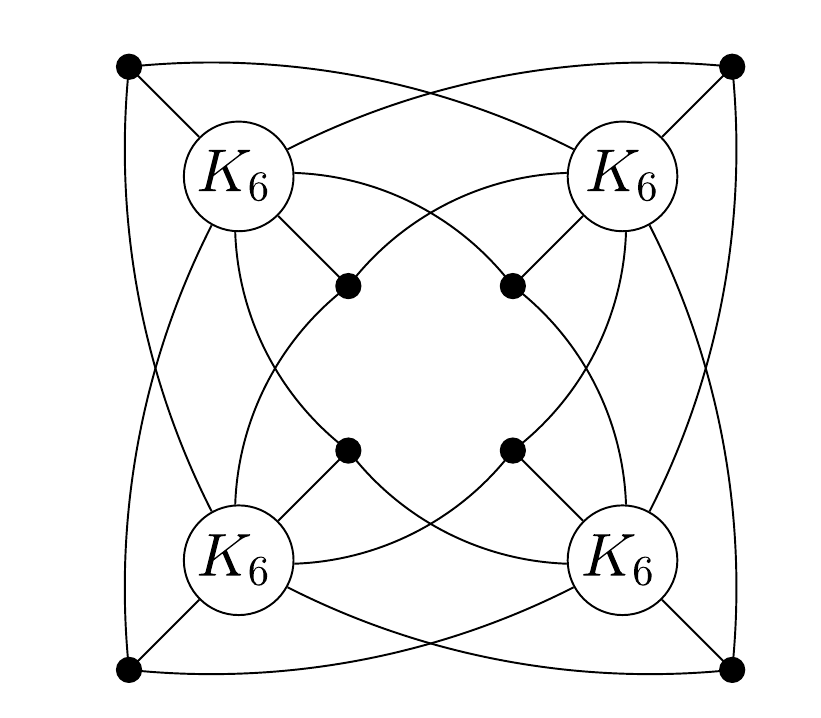}
\includegraphics[scale=0.4]{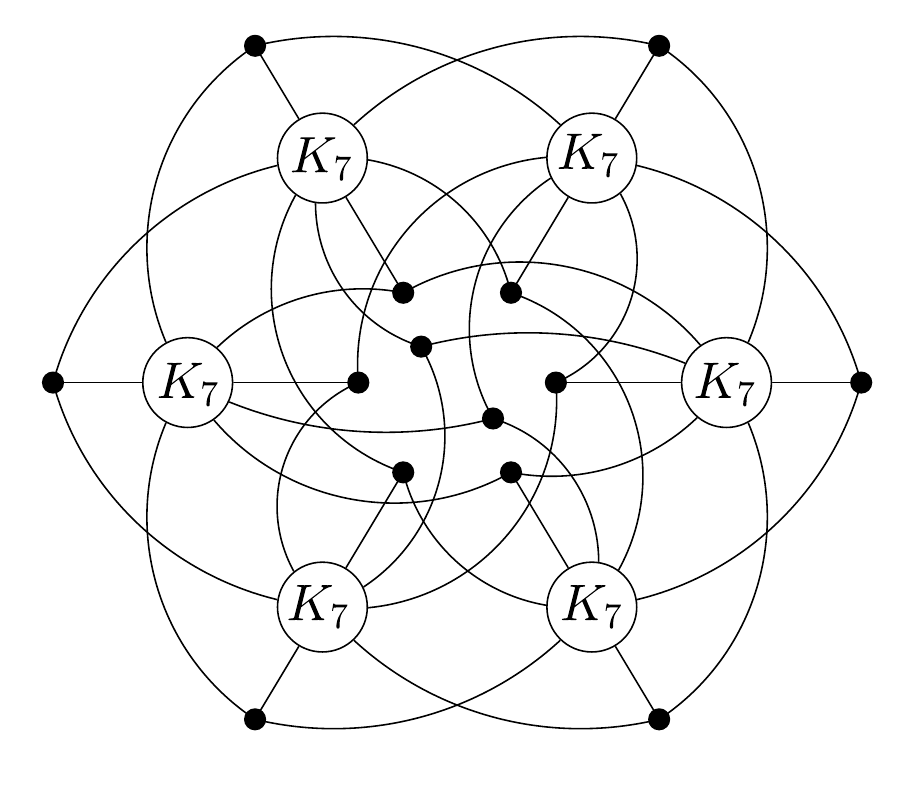}
\caption{$3$-connected essential $t$-connected non-rigid graphs for $t=6,7$}
\label{fig1}
\end{figure}

Let $n_3$ and $n_t$ denote the number of 3-vertex and $t$-vertex in $T$, respectively. Note that $3n_3=tn_t$, $n_3+tn_t=n(G)$ and it follows that $2n(G)-3=2\cdot 4n_3-3=8n_3-3$. Let $G'$ be a spanning subgraph of $G$. Then $G'$ contains at most $(2t-3)$ edges from each of the $t$-cliques. So when $t\le 8$ and $n>96$,
$$|E(G')|\leq3n_3 + n_t\cdot (2t-3)=9n_3-\frac9t n_3=\frac94(1-\frac1t)n(G')<2n(G')-3.$$ Then $G$ does not contain a spanning minimally rigid subgraph, that is, $G$ is not rigid.
	
Next, we will show that the graph $G$ is 3-connected and essentially $t$-connected. It is easy to see that $G-u-v$ is connected for all $u,v\in V(G)$. Let $S\in V(G)$ with $|S|\leq t-1$. Note that any vertex of $K_t$ is a $t$-vertex and adjacent to $t-1$ vertices of degree $t$ and one 3-vertex. Thus, there must be a $t$-vertex $u$ of any $K_t$ is adjacent to a 3-vertex denoted by $x$ in $G-S$. And then either $u$ is adjacent to other 3-vertices (except $x$) or $x$ is adjacent to other $t$-vertices (except $u$). We can extend this nontrivial component until there are only isolated vertices in $G-S$. So we have that $G$ is 3-connected and essentially $t$-connected.

\section{Appendix: complete of Claim~\ref{c1}}

{\bf Claim~\ref{c1}:}
In a 3-connected essentially 9-connected graph $G$, every $3$-vertex is only adjacent to $6^{+}$-vertices.

\smallskip

\begin{proof} Let $x\in V(G)$ be a $3$-vertex and $N(x)=\{v_1,v_2,v_3\}$. Suppose that $v_1$ is a $5^-$-vertex and $d(v_3), d(v_2)\ge d(v_1)$. We will simply use $d_i$ to denote $d(v_i)$ for i = 1,2,3, respectively. Let $N'(v_1)=N(v_1)-\{x,v_2,v_3\}$. Then let $S=\{v_2,v_3\}\cup N'(v_1)$ and $|S|=s$.  Note that $s\le 6$ and one component of $G-S$ contains edge $xv_1$. Let $U=\{x,v_1\}$. We also have $s\ge 3$, since $G$ is 3-connected. As $G$ has no essential cut of size less than $9$, the other components of $G-S$ must form an independent set, say $T=\{u_1, \ldots, u_t\}$.  Note that $N(u_i)\subseteq S$. We will find a spanning minimally rigid subgraph $H^\ast$ of $G-e$ for every $e\in E(G)$, which implies that $G$ is not a counterexample.

First of all, we claim that $s\ge4$. If not, then $s=3$ and $v_1$ has at least one neighbor in $\{v_2,v_3\}$ by $d(v_1)\ge3$. We know that all vertices in $T$ are adjacent to $S$ as $d(v)\ge3$ for $v\in V(G)$. Since $|V(G)|\ge10$, we have that $t\ge10-2-3=5$. It follows that there is a $K_{3,3}$ between $S$ and $T$ in $G-e$, for any $e\in E(G)$. We will construct $H^\ast$ from this $K_{3,3}$ by adding remaining vertices one by one and joining each vertex to its (any) two neighbors in current $H^\ast$ in $G-e$. Denote $K_{3,3}\cap T=\{u_1,u_2,u_3\}$.
Then $u_4,\ldots,u_t\in V(H^\ast)$, since $u_i$ has at least two neighbors in $K_{3,3}$ for $i\in\{4,\ldots,t\}$ in $G-e$. If $e$ is not incident with $x$, then we can add $x$ and $xv_2,xv_3$ to $H^\ast$;  and add $v_1$ and two incident edges other than $e$. If $e$ is incident with $x$, then we can add $v_1$ and two edges in $E[v_1,S]$ to $H^\ast$; and add $x$ and two incident edges incident other than $e$. Thus $H^\ast$ is a spanning subgraph of $G-e$ and
$|E(H^\ast)|=3\cdot3+2(n-6)=2n-3$. As $|E(K_{3,3})|=9=2|V(K_{3,3})|-3$, for any subgraph $H'$ of $H^\ast$ with $|H'|\ge2$ and $H^0=H'\cap K_{3,3}$,
then $$|E(H')|\le2|V(H'-H^0)|+|E(H^0)|\le2|V(H'-H^0)|+2|V(H^0)|-3=2|V(H')|-3,$$
which implies that $H^\ast$ is a spanning minimally rigid subgraph of $G-e$ and so $G$ is not a counterexample.
This completes the proof of the claim that $s\ge4$..

\bigskip
Let $T_0=N(v_2)\cap N(v_3)\cap T$ and $|T_0|=t_0$. 
For $s\ge4$, we take the spanning subgraph $H$ of $G$ that consists of the following types of edges.
\begin{description}
                            \item[type 1] all edges incident with $v_i$ in $G$, for $i\in[3]$.
                            \item[type 2] all edges between $S$ and $T$ in $G$.
                          \end{description}
We will prove several conclusions.

\begin{conclusion}\label{con1}
$|E(H)|\ge2n-2$.
\end{conclusion}

\begin{proof}
Suppose to the contrary that $|E(H)|\le 2n-3=2s+2t+1$. Recall that $4\le s\le 6$.
We first consider the case $s=4$. 
As $|V(G)|\ge10$, $t\ge10-2-4=4$. Note that $$|E_H[U,S]\cup E(U)|= d_1+2,$$ and $$2s+2t+1\ge|E(H)|\ge |E[S,T]|+|E_H[U,S]\cup E(U)|\ge3t+(d_1+2)\ge3t+5.$$ It implies that $d_1=3,t=4$ and $d(v)=3$ for any $v\in T$. It follows that $|E_H(S)|=0$ from the simple computation. Then any $v\in T$ is not adjacent to one vertex $v'\in S$ since $s=4$. If $v'\in N'(v_1)$, then consider $C=x\cup(N'(v_1)-v')\cup(T-v)$ and $|C|=5$. Since $E[\{v_2,v_3\},v']=\emptyset$, $\{v,v_2\}$ and $\{v_1,v'\}$ belong to different nontrivial components of $G-C$, a contradiction. If $v'\in \{v_2,v_3\}$, then consider $C=v_1\cup(\{v_2,v_3\}-v')\cap(T-v)$. Note that $|C|=5$ and $G-C$ contains two nontrivial components $\{x,v'\}$ and $v\cup N'(v_1)$ since $E[\{v_2,v_3\},N'(v_1)]=\emptyset$, a contradiction.

Next, we will show three useful facts before considering $5\le s\le 6$.
\begin{description}
\item[Fact 1] $t+|E_H(S)|\le s$.
\end{description}
\begin{proof}
Clearly, $|E_H[U,S]\cup E(U)|= d_1+2$ and $s\le d_1+1$.
According to types 1 and 2, we have that $2s+2t+1\ge|E(H)|\ge |E_H[U,S]\cup E(U)|+|E_H(S)|+3|T|=(d_1+2)+|E_H(S)|+3t.$
A direct computation shows that $t+|E_H(S)|\le 2s+2t+1-(d_1+2)-2t\le s.$
\end{proof}

\begin{description}
\item[Fact 2] There exists a vertex $v_i$ with $i\in\{2,3\}$ such that $|N(v_i)\cap N'(v_1)|\le2$.
\end{description}
\begin{proof}
Assume for a contradiction that both $v_2$ and $v_3$ have at least three neighbors in $N'(v_1)$, then $|E_H(S)|\ge6$. By Fact 1, we know that $|E_H(S)|=s=6$ and $t=0$. It implies that $|V(G)|=2+s+t=8$, a contradiction.
\end{proof}

\begin{description}
\item[Fact 3]  Suppose that $|V(G)|\le 12$. If $|N(v_i)\cap N'(v_1)|\le2$ and $v_iv_1\notin E(G)$ for some $i\in\{2,3\}$, then every vertex in $N(v_i)\cap T$ must be adjacent to all vertices in $N'(v_1)-N(v_i)$.
\end{description}
\begin{proof}
Without loss of generality we may assume that $i=2$. Suppose to the contrary that there is a vertex $u\in N(v_2)\cap T$ and a vertex $v\in N'(v_1)-N(v_2)$ such that $uv\notin E(G)$. Let $C=V(G)-\{v_1,v,u,v_2\}$. Note that $|C|\le12-4=8$ and $vv_1,uv_2\in E(G)$. Clearly, $\{v,v_1\}$ and $\{u,v_2\}$ belong to different nontrivial components of $G-C$, which contradicts the fact that $G$ is essentially 9-connected.
\end{proof}

Now we are ready to deal with the case $5\le s\le 6$.

When $s=5$, we know that $4\le d_1\le 5$. Note that $|V(G)|=2+s+t\le2+2s=12$ by Fact 1. 
Suppose that there is a vertex in $\{v_2,v_3\}$, say $v_2$, such that $v_2$ has no neighbors in  $N'(v_1)\cup v_1$. By Fact 3, every vertex in $N(v_2)\cap T$ must be adjacent to all vertices in $N'(v_1)$. Note that $|N'(v_1)|=3$. Thus, the neighbors of $v_2$ in $T$ are $4^+$-vertices. Note that $d_2\ge d_1\ge 4$ and $|N(v_2)\cap S|\le1$. So $|N(v_2)\cap T|\ge2$. If $|N(v_2)\cap T|\ge3$, then 
$|E(H)|\ge|E[U,S]\cup E(U)|+3|T|+(4-3)\cdot3\ge (d_1+2)+3t+3\ge 2s+2t+2$, 
a contradiction.
If $|N(v_2)\cap T|=2$, then $|N(v_2)\cap S|=1$ as $d(v_2)\ge4$ and $|E_H(S)|\ge1$ by type 1. As $|V(G)|\ge10$, $t=|V(G)|-2-s\ge3$. Thus, 
$|E(H)|\ge|E[U,S]\cup E(U)|+|E_H(S)|+3|T|+(4-3)\cdot2\ge (d_1+2)+3t+3\ge 2s+2t+2$, 
a contradiction. Therefore, each of $\{v_2,v_3\}$ must have at least one neighbor in
$N'(v_1)\cup v_1$ and $|E_H(S)|+|E[U,S]\cup E(U)|\ge d_1+2+|E_H(S)|\ge8$.

Since $|E(H)|\le 2s+2t+1$, $3t\le|E[S,T]|\le (2s+2t+1)-8$. We get that $t\le3$ by a direct computation. Since $t\ge3$, we have $t=3$ and $|E[S,T]|=3t$, that is, every vertex in $T$ is a 3-vertex. Thus, $|E_H(S)|+|E[U,S]\cup E(U)|=8$. It follows that each of $\{v_2,v_3\}$ must have exactly one neighbor in $N'(v_1)\cup v_1$ and $v_2v_3\notin E(G)$. Note that $v_1$ is adjacent to at most one vertex in $\{v_2,v_3\}$. There exists a vertex in $\{v_2,v_3\}$, say $v_2$, such that every vertex in $N(v_2)\cap T$ is adjacent to all vertices in $N'(v_1)-N(v_2)$ and $|N'(v_1)-N(v_2)|\ge2$ by Fact 3. Thus, each vertex in $T_0$ is a $4^+$-vertex. Since $d_3,d_2\ge d_1\ge4$, we have that $t_0=d_2+d_3-4-t\ge1$. That is, there is at least one vertex $u$ in $T_0$ of degree at least 4, a contradiction. 

When $s=6$, we have that $d_1=5$ and $v_1$ is not adjacent to $v_2$ and $v_3$.

We claim that $t\ge5$. Suppose to the contrary that $t\le4$. Then $|V(G)|=2+s+t\le12$. By Fact 2, there is a vertex $v_i, i\in\{2,3\}$, say $v_2$, such that $|N(v_2)\cap N'(v_1)|\le2$. According to Fact 3, every vertex in $N(v_2)\cap T$ is adjacent to all vertices in $N'(v_1)-N(v_2)$. Thus, the degree of each vertex in $T_0$ is at least 4. Note that $$d_2+d_3-2=|N(v_2)\cap N'(v_1)|+|N(v_3)\cap N'(v_1)|+2|N(v_2)\cap v_3|+|N(v_2)\cap T|+|N(v_3)\cap T|.$$
Therefore, we have that 
\begin{align*}
2s+2t+1&\ge|E(H)|=|E[U,S]\cup E(U)|+|E_H(S)|+|E[S,T]|\\
&\ge(d_1+2)+|N(v_2)\cap N'(v_1)|+|N(v_3)\cap N'(v_1)|+|N(v_2)\cap v_3|+3|T|+|T_0|\\
&\ge(d_1+2)+(d_2+d_3-2-1)-(|N(v_2)\cap T|+|N(v_3)\cap T|)+3|T|+|T_0|\\
&=d_1+d_2+d_3+2t-1,
\end{align*}
since $t+t_0\ge |N(v_2)\cap T|+|N(v_3)\cap T|$. As $d_i\ge5,i\in[3]$ and $s=6$, we get $13\ge15-1$ by the computation, a contradiction.
This completes the proof of the claim that $t\ge 5$.

\smallskip
By Fact 1, we have that $5\le t\le6$. We claim that $v_i$ must have at least one neighbor in $N'(v_1)$ for any $i\in\{2,3\}$.
Suppose to the contrary that there is a vertex $v_i,i\in\{2,3\}$, say $v_2$, such that $N(v_2)\cap N'(v_1)=\emptyset$.

Let $W(v)=\{w: w\in N'(v_1)\ \&\ vw\notin E(G),\ \emph{for}\ v\in N(v_2)\cap T\}.$ We show that $|W(v)|\le t-4$, for any $v\in N(v_2)\cap T$. If not, then there is a vertex $u\in N(v_2)\cap T$ such that $|W(u)|\ge t-3$. Consider $C=V(G)-(\{v_1,v_2,u\}\cup W(u))$. Note that $|C|\le2+s+t-(3+t-3)=8$. Clearly, $v_1\cup W(u)$ and $\{u,v_2\}$ belong to different components of $G-C$. Both of them contain at least one edge, a contradiction to the fact that $G$ is essentially 9-connected. Thus $|W(v)|\le t-4$.

Therefore, we have that the degree of each vertex $v$ in $N(v_2)\cap T$ is at least $|N'(v_1)|-|W(v)|+1\ge 9-t.$
The degree of each vertex in $T_0$ is at least $10-t.$

Now we show that $t_0 \ge 1$. By Fact 1, $|E_H(S)|\le s-t=6-t$. Since $N(v_2)\cap N'(v_1)=\emptyset$, $|N(v_3)\cap S|=|E_H(S)|$.
Thus, $|N(v_3)\cap T|\ge d_3-1-|N(v_3)\cap S|\ge t+d_3-7.$
Therefore, $$t_0\ge|N(v_2)\cap T|+|N(v_3)\cap T|-t\ge (d_2-2)+(t+d_3-7)-t=d_2+d_3-9\ge 1.$$
Thus, 
$$|E(H)|\ge|E[U,S]\cup E(U)|+3|T|+(10-t-3)\cdot t_0\ge d_1+2t+9=2s+2t+2.$$
This leads to a contradiction, completing the proof of the claim that $v_i$ has at least one neighbor in $N'(v_1)$ for any $i\in\{2,3\}$.

It implies that $|E_H(S)|\ge2$. So $$|E(H)|\ge|E[U,S]\cup E(U)|+|E_H(S)|+3|T|\ge (d_1+2)+2+3t\ge2s+2t+2,$$ a contradiction.
This completes the proof of Conclusion~\ref{con1}.
\end{proof}

\medskip
Thus $H$ is a spanning subgraph of $G$ with $|E(H)|\ge 2n-2$, and we can find a spanning subgraph $H_1$ of $H-e$ with $|E(H_1)|=2n-3$ for any $e\in E(G)$. If $H_1$ is a spanning minimally rigid subgraph of $G-e$, then we are done. So we assume that $H_1$ is not a spanning minimally rigid subgraph of $G-e$, that is, there is a subgraph $Z\subseteq H_1$ of size at least 2 such that $|E(Z)|>2|V(Z)|-3$. Choose $Z\subseteq H_1$ such that $Z$ is a minimal subgraph. That is, $|E(Z)|>2|V(Z)|-3$. But for any proper subgraph $Z'$ of $Z$, $|E(Z')|\le2|V(Z')|-3$. Clearly, $|V(Z)|\ge4$ and $d_Z(v)\ge3$ for each $z\in V(Z)$.
\begin{conclusion}\label{con2}
$Z$ must satisfy one of the following conditions.\\
\emph{(a)} $|V(Z)\cap S|\ge3$.\\
\emph{(b)} $V(Z)=\{x,v_1,v_2,v_3\}$ and $Z$ is a $K_4$.
\end{conclusion}
 \begin{proof}
Suppose that $|V(Z)\cap S|\le2$. Then $|N(v)\cap Z|\le2$ for any $v\in T$, which implies that $V(Z)\cap T=\emptyset$ by the minimality of $Z$.
As $|V(Z)|\ge4$ and $|V(Z)\cap S|\le2$, $U\subseteq V(Z)$ and $Z$ must be a $K_4$. Since $x$ is not adjacent to $N'(v_1)$, $N'(v_1)\nsubseteq V(Z)$ and $V(Z)=\{x,v_1,v_2,v_3\}$.
 \end{proof}


We pick a spanning subgraph $Z_0$ of $Z$ such that $|E(Z_0)|=2|V(Z_0)|-3$. A Laman graph (minimally rigid graph) $G$ is \textbf{\emph{maximal}} if no Laman graph properly contains $G$. 
We will show that we can obtain a spanning minimally rigid subgraph $H^\ast$ of $G-e$ by adding the remaining vertices of $G$ and some edges in $E(G)$ to $Z_0$. 
Firstly, if $x_i\in V(G)-V(Z_{i-1})$ and $x_i$ has at least two neighbors in $Z_{i-1}$ in $G-e$, 
then add vertex $x_i$ and two edges (other than $e$) between $x_i$ and $Z_{i-1}$ to $Z_{i-1}$. 
Denote by $Z_i$ the resulting graph. So $V(Z_i)=V(Z_{i-1})\cup x_i$ and $|E(Z_i)|=|E(Z_{i-1})|+2$. Suppose that $i\in[k]$, that is, we can add at most $k$ vertices to $Z_0$.

Clearly, $|E(Z_k)|=|E(Z_0)|+2k=2|V(Z_k)|-3$.
For any subgraph $H'\subseteq Z_k$ with $|H'|\ge2$, we have that $|E(H')|\le 2|V(Z_0\cap H')|-3+2(|V(H')|-|V(Z_0)|)\le 2|V(H')|-3$.
For convenience, let $S'=V(Z_k)\cap S$ and $T'=V(Z_k)\cap T$. Let $S-S'=\{w_1,w_2,\ldots\}$ and $T-T'=\{u_1,u_2,\ldots\}$.

\begin{conclusion}\label{con3}
If $S'\nsubseteq N'(v_1)$ and $U\nsubseteq Z_k$, then the following hold.\\
\emph{(a)} $e\in E[U,S']\cup E(U)$.\\
\emph{(b)} Let $U=\{x_1,x_2\}$. If $x_1\in U$ has at least two neighbors in $S'$, then $x_2\notin Z_k$ and $x_2$ has at most one neighbor in $S'$.
\end{conclusion}
\begin{proof}
By Conclusion \ref{con2}, $|S'|\ge |V(Z)\cap S|\ge3$. \\
(a) If $e\notin E[U,S']\cup E(U)$, then there is a vertex say $x_1\in U$ that has at least two neighbors in $S'$ and $x_1\in Z_k$. Another vertex $x_2\in U$ must have at least two neighbors in $S'\cup x_1$ and $x_2\in Z_k$, a contradiction.\\
(b) If $x_2\in Z_k$, then $|N(x_1)\cap (x_2\cup S')|\ge 3$ and $x_1\in Z_k$, a contradiction. Thus, $x_2\notin Z_k$. If $|N(x_2)\cap S'|\ge2$, then $e\in E[x_2,S']$, as $x_2\notin Z_k$. So in $G-e$, $x_1$ has at least two neighbors in $S'$ and $x_1\in Z_k$. Since $x_2x_1\in E(G-e)$ and $|E[x_2,S']-e|\ge1$, we have that $x_2\in Z_k$, a contradiction.
\end{proof}

Next, we construct $Z_{k'}$, when $U\nsubseteq Z_k$ and $S'\nsubseteq N'(v_1)$. Note that $|S'|\ge3$ by Conclusion \ref{con2}. Let $U'=U-Z_k$. In $G$, there must exist a vertex, say $x_1\in U$ that has at least two neighbors in $S'$ and another vertex, say $x_2\in U$ that has at least one neighbor in $S'$. By Conclusion \ref{con3}(b), $x_2\in U'$. If $U'=U$, then add $x_1,x_2$ and $e_1,e_2\in E[x_1, S'], e_3\in E[x_2,S'],x_1x_2$ to $Z_k$ and delete $e$ to get $Z_{k'}$. If $U'=x_2$, then we add $x_2$ and $x_1x_2, e_3\in E[x_2, S']$ and delete $e$ to get $Z_{k'}$. Therefore, we have that $|E(Z_{k'})|=2|V(Z_{k'})|-4.$

\begin{conclusion}\label{con4}
Let $S_0=\{w: w\in S-S'\ \&\  N(w)\cap (Z_k-U)=\emptyset\}$ and $S_1=\{w: w\in S-S'\ \&\  |N(w)\cap (Z_k-U)|=1\}$. If $S'\nsubseteq N'(v_1)$ and $U\nsubseteq Z_k$, then $|S_0|\ge2$. 
\end{conclusion}
\begin{proof}
By Conclusion \ref{con3}, $e\in E[U,S']\cup E(U)$ and $S-S'\ne\emptyset$. So $|N(w)\cap (Z_k-U)|\le1$ for each $w\in S-S'$ and $S_0\cup S_1=S-S'$. Suppose to the contrary that $|S_0|\le1$. If $S_0=\{w_1\}$, then we add each vertex  $w\in S_1$ and its two incident edges $e_1\in E[w,U],e_2\in E[w,Z_k-U]$ to $Z_{k'}$. Then add the remaining vertices of $T$ one by one by joining it with two neighbors in the existing subgraph. At last, add $w_1$ and three incident edges including $e'\in E[w_1,U]$. If $S_0=\emptyset$, then choose $w_1\in S_1$ and add each vertex $w\in S_1-w_1$ and two incident edges $e_1\in E[w,U],e_2\in E[w,Z_k-U]$ to $Z_{k'}$. Then add the remaining vertices of $T$ one by one by joining it with two neighbors in the already-formed subgraph. At last, add $w_1$ and three incident edges including $e'\in E[w_1,U]$. Note that $|E(Z_{k'})|=2|V(Z_{k'})|-4$. Thus, $$|E(H^\ast)|=|E(Z_{k'})|+2|V(G)-V(Z_k)|+1=2|V(G)|-3.$$

We will show that $|E(H')|\le 2|V(H')|-3$ for any subgraph $H'\subseteq H^\ast$ with $|H'|\ge2$, to reach a contradiction.
Clearly, if $H'\subseteq Z_k$, then $|E(H')|\le 2|V(H')|-3$. Suppose that  $V(H')\subseteq V(H^\ast)-V(Z_k)$.
If $U'\cap V(H')=\emptyset$, then
\begin{align*}
|E(H')|&\le |V(H')\cap(S-S')|\cdot |V(H')\cap(T-T')|+|E[S-S']|\\
&\le\max\{|V(H')|-1,2(|V(H')|-2)+1,2(|V(H')|-3)+2\}\\
&\le2|V(H')|-3.
\end{align*}
If $|U'\cap V(H')|=1$, then
\begin{align*}
|E(H')|&\le |V(H')\cap(S-S')|\cdot |V(H')\cap(T-T')|+|E[S-S']|\\
&+|E[U'\cap V(H'), (S-S')\cap V(H')]|\\
&\le\max\{(|V(H')|-2)+1,2(|V(H')|-3)+1+2,2(|V(H')|-4)+2+3 \}\\
&\le2|V(H')|-3.
\end{align*}
If $|U'\cap V(H')|=2$, then
\begin{align*}
|E(H')|&\le |V(H')\cap(S-S')|\cdot |V(H')\cap(T-T')|+|E[S-S']|\\
&+|E[U'\cap V(H'), (S-S')\cap V(H')]|+|E(U')|\\
&\le\max\{(|V(H')|-3)+1,2(|V(H')|-4)+1+2,2(|V(H')|-5)+2+3 \}+1\\
&\le2|V(H')|-3.
\end{align*}
Next, we consider that $V(H')\cap V(Z_k)\ne\emptyset$ and $V(H')\cap(V(H^\ast)-V(Z_k))\ne\emptyset$. Note that $Z_k$ is the unique maximal Laman subgraph in $H^\ast-w_1$. Since $|N(w_1)\cap Z_k|\le1$, we have that $|E(H')|\le2|V(H')|-3$. This completes the proof of Conclusion~\ref{con4}.
\end{proof}

\begin{conclusion}\label{con5}
$2\le|S-S'|\le3$.
\end{conclusion}
\begin{proof}
By Conclusion \ref{con2}, if $6\ge s\ge5$, then $|S'|\ge|V(Z)\cap S|\ge3$ and $|S-S'|\le3$. If $s=4$, then $|S'|\ge|V(Z)\cap S|\ge2$ and $|S-S'|\le2$. Suppose to the contrary that $|S-S'|\le1$.  Then all but at most one vertex in $T$ are adjacent to at least two vertices in $S'$ in $G-e$, that is, $|T-T'|\le1$.

If $T-T'=u_1$, then $u_1$ is adjacent to $w_1\in S-S'$ and is incident with $e$ in $G$ as $d(u_1)\ge3$. Note that $|N(w_1)\cap Z_k|\le1$ since $w_1\in S-S'$ and $w_1$ is not incident with $e$. It implies that $|N(w_1)\cap (V(G)-V(Z_k))|\ge2$ and $U\nsubseteq Z_k$. By Conclusion \ref{con3}(a), $e\in E(U)\cup E[U,S']$, a contradiction to the fact that $e$ is incident with $u_1$. So we assume that  $T=T'$. If $S=S'$, then $U\subseteq Z_k$ and $V(Z_k)=V(G)$. Note that $Z_k$ is a spanning minimally rigid subgraph of $G-e$, a contradiction.
If $S-S'=w_1$, then $|N(w_1)\cap (V(G)-V(Z_k))|\ge1$ and $U\nsubseteq Z_k$. Note that $S'\nsubseteq N'(v_1)$, we have that $e\in E(U)\cup E[U,S']$ by Conclusion \ref{con3}(a). As $e$ is not incident with $w_1$, $|N(w_1)\cap (V(G)-V(Z_k))|\ge2$. It follows that $x,v_1\notin Z_k$ and $w_1$ is adjacent to $x$ and $v_1$. So $d(w_1)=3$ and $w_1\in\{v_2,v_3\}$. As $|S|\ge4$ and $N(v_1)\cap\{v_2,v_3\}\ne\emptyset$, $d(v_1)\ge4$, a contradiction to the fact that $d(v_2),d(v_3)\ge d(v_1)$.
\end{proof}

It is not hard to see that $d(v_1)\ge 4$. Otherwise if $d(v_1)=3$, since $s\ge4$, it follows that $s=4$ and $v_1$ is not adjacent to $v_2$ and $v_3$.
By conclusion \ref{con2}, $|S'|\ge|V(Z)\cap S'|\ge3$ and $|S-S'|\le1$, a contradiction to Conclusion \ref{con5}. It implies that $3$-vertices are only adjacent to $4^+$-vertices.

\bigskip
To complete the proof, we now construct a spanning minimally rigid subgraph $H^\ast$ by considering the following cases.

\noindent\textbf{Case 1:} $S'\nsubseteq N'(v_1)$.

\textbf{Subcase 1.1:} $S-S'\subseteq N'(v_1)$ and $|S-S'|=2$.

(\romannumeral1) Suppose that $U\subseteq V(Z_k)$.

We claim that $|T-T'|\ge2$. If $e\notin E[S-S',Z_k]$, then $|T-T'|\ge2$ as $N(S-S')\subseteq v_1\cup(T-T')$ and $G$ is 3-connected. If $e\in E[S-S',Z_k]$, then $|N(S-S')\cap (Z_k-v_1)|\le1$ and $|T-T'|\ge1$ as $N(S-S')\subseteq Z_k\cup(T-T')$ and $G$ is 3-connected. As $u_1$ is not incident with $e$, $|N(u_1)\cap S'|\le1$ for $u_1\in T-T'$. It follows that $d(u_1)=3$ and $u_1$ is adjacent to $w_1$ and $w_2$. Since $3$-vertices are only adjacent to $4^+$-vertices, $w_1$ and $w_2$ are $4^+$-vertices. There exists a vertex in $S-S'$, say $w_1$, that is not incident with $e$. As $N(w_1)\subseteq \{v_1,w_2\}\cup(T-T')$, we get that $|T-T'|\ge2$.

Let $N'(v_i)=N(v_i)\cap (T-T')$ and $N''(v_i)=N(v_i)\cap(S-S')$ for $i\in\{2,3\}$. We show that $|N'(v_i)|\ge2$ for any $i\in\{2,3\}$. Suppose to the contrary that there is a vertex, say $v_2$, such that $|N'(v_2)|\le1$. Consider $C=v_1\cup N'(v_2)\cup(S'-v_2)\cup N''(v_2)$. Note that $|C|\le6$ as $|N''(v_2)|\le1$. We have that $(S-S'-N''(v_2))\cup(T-T'-N'(v_2))$ and $\{x,v_2\}$ lie in different components of $G-C$ and $xv_2\in E(G)$. As $|T-T'|\ge2$, there is at least one edge in $(S-S'-N''(v_2))\cup(T-T'-N'(v_2))$. So $C$ is an essential cut, a contradiction. Since all but at most one vertex in $T-T'$ can be only adjacent at most one vertex in $S'$, we have that $|T-T'|\ge3$ and there are three different vertices $u_1,u_2$ and $u_3$ in $T-T'$ such that $u_1v_2,u_2v_3,u_3v_3\in E(G)$ and $u_i$ is adjacent to $w_1,w_2$ for $i\in[2]$. We add $w_i$ and $v_1w_i$ to $Z_k$ and obtain $Z_{k+2}$, where $i\in\{1,2\}$. Then for $i\in\{1,2\}$, add $u_i$ and $u_iw_1,u_iw_2,u_iv_{i+1}$ to $Z_{k+2}$ and obtain $Z_{k+4}$.

If $e\notin E(Z_{k+4})$, then add remaining vertices in $T-T'$ one by one by joining it with two vertices in the already-formed subgraph in $G-e$ to get $H^\ast$. Thus, we have that $$|E(H^\ast)|=|E(Z_{k+4})|+2(|T-T'|-2)=2|V(G)|-3.$$

If $e\in E(Z_{k+4})$, then $u_3$ is adjacent to $w_1$ and $w_2$. We add $u_3$ and $u_3w_1,u_3w_2,u_3v_3$ to $Z_{k+4}$ and obtain $Z_{k+5}$. And then add remaining vertices in $T-T'$ one by one by joining it with two vertices in the already-formed subgraph in $G-e$ and get $Z_{n'}$. And delete $e$ from $Z_{n'}$ to get $H^\ast$. Thus, we have that $$|E(H^\ast)|=|E(Z_{k+5})-e|+2(|T-T'|-3)=2|V(G)|-3.$$

It is remaining to show that $|E(H')|\le 2|V(H')|-3$ for any subgraph $H'\subseteq H^\ast$ with $|H'|\ge2$.
Clearly, if $H'\subseteq Z_k$, then $|E(H')|\le 2|V(H')|-3$.  If $V(H')\subseteq V(H^\ast)-V(Z_k)$, then
\begin{align*}
|E(H')|&\le |H'\cap(S-S')|\cdot |H'\cap(T-T')|\\
&\le\max\{|V(H')|-1,2(|V(H')|-2)\}\\
&\le2|V(H')|-3.
\end{align*}
Next, we consider that $V(H_1')=V(H')\cap V(Z_k)\ne\emptyset$ and $V(H'_2)=V(H')-V(H'_1)\ne\emptyset$. Let $W=\{w_1,w_2,u_1,u_2,u_3\}$. It is easy to check that for any subset $W'\subseteq W$ and any subgraph $Z'\subseteq Z_k$, $|E_{H^\ast}(W'\cup V(Z'))|-|E(Z')|\le 2|W'|.$
Each vertex in $V(H^\ast)-V(Z_k)-W$ is a 2-vertex in $H^\ast$. Thus, every time we add any subset $Y\subseteq V(H^\ast)-V(Z_k)$ into $Z_k$, we have that $|E_{H^\ast}(Y\cup V(Z'))|-|E(Z')|\le 2|Y|$ for any subgraph $Z'\subseteq Z_k$. Thus, $|E(H')|\le |E(H_1')|+2|V(H_2')|\le2|V(H')|-3$.


\smallskip
(\romannumeral2) Suppose that $U\nsubseteq V(Z_k)$.

As $U\nsubseteq Z_k$, $|S'|\ge3$ by Conclusion \ref{con2}. So $|S|\ge5$. Since $|N(x)\cap S'|\ge2$, $v_1\notin Z_k$ and $|N(v_1)\cap S'|=1$ by Conclusion \ref{con3}(b). So $|S'|=3$. As $|N(x)\cap Z_k|=2$, $x\notin Z$. It follows that $|T\cap V(Z)|\ge1$ from $|V(Z)|\ge4$. Since $d_Z(v)\ge3$ for any $v\in Z$, all vertices in $T\cap V(Z)$ are adjacent to all vertices in $S'$. According to Conclusion \ref{con4}, $|S_0|\ge2$ and so $S-S'=S_0$. Let $N(v_1)\cap S'=z_1$ and $S_0=\{w_1,w_2\}$. We claim that $|T-T'|\ge6$. Suppose to the contrary that $|T-T'|\le5$. Then $\{x,z_1\}\cup (T-T')$ is an essential cut of size at most 7, as $\{v_1,w_1,w_2\}$ and $(S'-z_1)\cup T'$ belong to different nontrivial components, a contradiction.

Let $N'(v_i)=N(v_i)\cap (T-T')$ for $i\in\{2,3\}$. Next, we show that $|N'(v_i)|\ge4$ for $i\in\{2,3\}$. Suppose to the contrary that there is a vertex, say $v_2$, such that $|N'(v_2)|\le3$. Then $C=\{x,v_1\}\cup N'(v_2)\cup(S-S_0-v_2)$ is an essential cut of size at most 8, as $S_0\cup(T-T'-N'(v_2))$ and $v_2\cup T'$ belong to different nontrivial components, a contradiction. As $e\in E(U)\cup E[U,S']$, in $G-e$, there are three vertices $u_1,u_2$ and $u_3$ in $T-T'$ such that $u_1v_2,u_2v_3,u_3v_2\in E(G)$ and $u_i$ is adjacent to $w_1$ and $w_2$ for $i\in[3]$.

Note that $|E(Z_{k'})|=2|V(Z_{k'})|-4.$ We will construct $H^\ast$ in the following steps.
First, we add $w_i$ and $v_1w_i$ to $Z_{k'}$ and obtain $Z_{k'+2}$, where $i\in\{1,2\}$. We then add $u_i$ and $u_iw_1,u_iw_2,u_iv_{i+1}$ to $Z_{k'+2}$ for $i\in[2]$ and obtain $Z_{k'+4}$. In the last step, we add $u_3$ and $u_3w_1,u_3w_2,u_3v_2$ to obtain $Z_{k'+5}$, and add remaining vertices in $T-T'$ one by one by joining it with two neighbors in $V(Z_{k'+5})$ to get $H^\ast$.  Thus, we have that $$|E(H^\ast)|=|E(Z_{k'})|+2+3\cdot3+2(|T-T'|-3)=2|V(G)|-3.$$
Let $W=U'\cup\{w_1,w_2,u_1,u_2,u_3\}$. It is easy to check that for any subset $W'\subseteq W$ and any subgraph $Z'\subseteq Z_k$,
$|E_{H^\ast}(W'\cup V(Z'))|-|E(Z')|\le 2|W'|.$

It is remaining to show that $|E(H')|\le 2|V(H')|-3$ for any subgraph $H'\subseteq H^\ast$ with $|H'|\ge2$.
Clearly, if $H'\subseteq Z_k$, then $|E(H')|\le 2|V(H')|-3$.  If $V(H')\subseteq V(H^\ast)-V(Z_k)$, then
\begin{align*}
|E(H')|&\le |V(H')\cap(S-S')|\cdot |V(H')\cap(T-T')|\\
&+|E[U'\cap V(H'),(S-S')\cap V(H')]|+|E(U'\cap V(H'))|\\
&\le\max\{|V(H')|-1,2(|V(H')|-2),(|V(H')|-2)+1,2(|V(H')-3|)+2,\\&(|V(H')|-3)+1+1,2(|V(H')|-4)+2+1\}\\
&\le2|V(H')|-3.
\end{align*}
Next, we consider that $V(H_1')=V(H')\cap V(Z_k)\ne\emptyset$ and $V(H'_2)=V(H')-V(H'_1)\ne\emptyset$. 
Each vertex in $V(H^\ast)-V(Z_k)-W$ is a 2-vertex in $H^\ast$. Thus, every time we add any subset $Y\subseteq V(H^\ast)-V(Z_k)$ into $Z_k$, we have that  $|E_{H^\ast}(Y\cup V(Z'))|-|E(Z')|\le 2|Y|$ for any subgraph $Z'\subseteq Z_k$. Thus, $|E(H')|\le |E(H_1')|+2|V(H_2')|\le2|V(H')|-3$, completing the proof.

\medskip
\textbf{Subcase 1.2:} $|(S-S')\cap N'(v_1)|=1$ and $|(S-S')\cap \{v_2,v_3\}|=1$.

For convenience, suppose that $w_1\in N'(v_1)$ and $w_2=v_2$.

(\romannumeral1) Suppose that $U\subseteq V(Z_k)$.

We first show that $|T-T'|\ge2$. As there is a vertex, say $w_1$, in $S-S'$ that has only one neighbor in $Z_k$, $|T-T'|\ge d(w_1)-1-(|S-S'|-1)=1$. If $T-T'=u_1$, then $d(w_1)=3$ and $w_1$ is adjacent to $w_2$. It follows that $d(w_2),d(u_1)\ge4$ since 3-vertices are only adjacent to $4^+$-vertices. As $N(u_1)\subseteq S$, $|N(u_1)\cap S'|\ge2$ and so $e\in E[u_1,S']$. However, in this case, $w_2$ only can be adjacent to at most three vertices, a contradiction. Thus, $|T-T'|\ge2$.

If $|N(w_2)\cap V(Z_k)|=2$, then $e\in E[w_2,V(Z_k)]$ since $w_2\in S-S'$. As all vertices in $T-T'$ are adjacent to at most one vertex in $Z_k$, each vertex $u\in T-T'$ is adjacent to $w_1$ and $w_2$. We will construct $H^\ast$ in the following steps.
First add $w_1,w_2$ and $w_1v_1$, $E[w_2,V(Z_k)]-e$ to $Z_k$ to get $Z_{k+2}$. For $i\in[2]$, add $u_i\in T-T'$ and three edges incident with $u_i$ including $u_1w_1,u_2w_2$ to $Z_{k+2}$ and get $Z_{k+2+i}$. In the last step, we add remaining vertices in $T-T'$ one by one by joining it with two neighbors in $V(Z_{k+4})$ to get $H^\ast$.  Thus, we have that $$|E(H^\ast)|=|E(Z_{k})|+2+3\cdot2+2(|T-T'|-2)=2|V(G)|-3.$$ Let $W=\{w_1,w_2,u_1,u_2\}$.  It is easy to check that for any subset $W'\subseteq W$ and any subgraph $Z'\subseteq Z_k$, $|E_{H^\ast}(W'\cup V(Z'))|-|E(Z')|\le 2|W'|.$

\smallskip
Next, suppose that $N(w_2)\cap V(Z_k)=x$. Let $\{z_1,z_2\}\subseteq S'\cap N'(v_1)$ and $N'(z_i)=N(z_i)\cap (T-T')$ for $i\in[2]$. Then we claim that $|N'(z_i)|\ge2$ for any $i\in[2]$. Suppose to the contrary that there is a vertex, say $z_1$, such that $|N'(z_1)|\le1$. Consider $C=N'(z_1)\cup(S-z_1-w_2)\cup x$ and $|C|\le1+4+1=6$. Since all vertices in $w_2\cup (T-T'-N'(z_1))$ are not adjacent to all vertices in $\{v_1,z_1\}\cup T'$. Thus, $w_2\cup (T-T'-N'(z_1))$ and $\{v_1,z_1\}\cup T'$ belong to different components of $G-C$ and $v_1z_1\in E(G)$. As $|N'(z_1)|\le1$ and $|T-T'|\ge2$, there is at least one edge in $w_2\cup (T-T'-N'(z_1))$. So $C$ is an essential cut, a contradiction. Since all but at most one vertex in $T-T'$ can be only adjacent at most one vertex in $S'$, we have that $|T-T'|\ge3$ and there are three different vertices $u_1,u_2$ and $u_3$ such that $u_1z_1,u_2z_2,u_3z_1\in E(G)$ and $u_i$ is adjacent to $w_1$ and $w_2$ for $i\in[2]$.

We will construct $H^\ast$ in the following steps.
First add $w_1,w_2$ and $v_1w_1,xw_2$ to $Z_k$ and obtain $Z_{k+2}$. Then for $i\in[2]$, add $u_i$ and $u_iw_1,u_iw_2,u_iz_{i}$ to $Z_{k+2}$ and obtain $Z_{k+4}$.

If $e\notin E(Z_{k+4})$, then add remaining vertices in $T-T'$ one by one by joining it with two vertices in $Z_{k+4}$ in $G-e$ to get $H^\ast$. Thus, we have that $|E(H^\ast)|=|E(Z_{k+4})|+2(|T-T'|-2)=2|V(G)|-3.$ Let $W=\{w_1,w_2,u_1,u_2\}$. It is easy to check that for any subset $W'\subseteq W$ and any subgraph $Z'\subseteq Z_k$, $|E_{H^\ast}(W'\cup V(Z'))|-|E(Z')|\le 2|W'|.$

If $e\in E(Z_{k+4})$, then $u_3$ is adjacent to $w_1$ and $w_2$. We add $u_3$ and $u_3w_1,u_3w_2,u_3z_1$ to $Z_{k+4}$ and obtain $Z_{k+5}$.
In the last step, we add remaining vertices in $T-T'$ one by one by joining it with two vertices in $Z_{k+5}$ in $G-e$ to get $Z_{n'}$, and delete $e$ from $Z_{n'}$ to get $H^\ast$. Thus, we have that $|E(H^\ast)|=|E(Z_{k+5})-e|+2(|T-T'|-3)=2|V(G)|-3.$ Let $W=\{w_1,w_2,u_1,u_2,u_3\}$.  It is easy to check that for any subset $W'\subseteq W$ and any subgraph $Z'\subseteq Z_k$, $|E_{H^\ast}(W'\cup V(Z'))|-|E(Z')|\le 2|W'|.$

It is remaining to show that $|E(H')|\le 2|V(H')|-3$ for any subgraph $H'\subseteq H^\ast$ with $|H'|\ge2$.
Clearly, if $H'\subseteq Z_k$, then $|E(H')|\le 2|V(H')|-3$.  If $V(H')\subseteq V(H^\ast)-V(Z_k)$, then
\begin{align*}
|E(H')|&\le |V(H')\cap(S-S')|\cdot |V(H')\cap(T-T')|\\
&\le\max\{|V(H')|-1,2(|V(H')|-2)\}\\
&\le2|V(H')|-3.
\end{align*}
Next, we consider that $V(H_1')=V(H')\cap V(Z_k)\ne\emptyset$ and $V(H'_2)=V(H')-V(H'_1)\ne\emptyset$. Note that each vertex in $V(H^\ast)-V(Z_k)-W$ is a 2-vertex in $H^\ast$. Thus, every time we add any subset $Y\subseteq V(H^\ast)-V(Z_k)$ into $Z_k$, we have that $|E_{H^\ast}(Y\cup V(Z'))|-|E(Z')|\le 2|Y|$ for any subgraph $Z'\subseteq Z_k$. Thus, $|E(H')|\le |E(H_1')|+2|V(H_2')|\le2|V(H')|-3$.

\smallskip
(\romannumeral2) Suppose that $U\nsubseteq V(Z_k)$.

In this case, $|N(v_1)\cap S'|\ge2$ implies that $x\notin Z_k$ and $|N(x)\cap S'|=1$ by Conclusion \ref{con3}(b). We first show that $T\cap V(Z)\ne\emptyset$. Suppose to the contrary that $T\cap V(Z)=\emptyset$. Then $V(Z)\subseteq v_1\cup S'$. As $Z\subseteq H$, we know that every vertex in $N'(v_1)\cap V(Z)$ can only be adjacent to $v_1$ and $v_3$ in $H$. Since $d_Z(v)\ge3$ for each $v\in Z$, $N'(v_1)\cap V(Z)=\emptyset$, a contradiction to the fact that $|V(Z)\cap S|\ge3$.
Thus $|T\cap V(Z)|\ge1$ and each vertex in $T\cap V(Z)$ is adjacent to at least three vertices in $S'$. According Conclusion \ref{con4}, $|S_0|\ge2$ and so $S-S'=S_0$.
We also have that $|T-T'|\ge6$, for otherwise if $|T-T'|\le5$, then $\{v_1,w_1\}\cup (T-T')\cup v_3$ is an essential cut of size at most 8,
since $\{x,w_2\}$ and $(S'-v_3)\cup T'$ belong to different nontrivial components, a contradiction.

Let $N'(z)=N(z)\cap (T-T')$ for $z\in S'$. We show that there are at least two vertices $z_1,z_2\in N'(v_1)\cap S'$ such that $|N'(z_i)|\ge3$ for $i\in[2]$. Suppose to the contrary that there is at most one vertex, say $z_1$, such that $|N'(z_1)|\ge3$. Let $S^\ast= N'(v_1)\cap S'-z_1$. Then $1\le|S^\ast|\le2$. Thus, $|N'(S^\ast)|=\sum_{z\in S^\ast}|N'(z)|\le4$. Then $C=\{x,v_1\}\cup N'(S^\ast)\cup \{z_1,v_3\}$ is an essential cut of size at most 8, as $S_0\cup(T-T'-N'(S^\ast))$ and $S^\ast\cup T'$ belong to different nontrivial components, a contradiction. Therefore, there are three different vertices $u_1,u_2$ and $u_3$ in $T-T'$ such that $u_1z_1,u_2z_2,u_3z_1\in E(G)$ and $u_i$ is adjacent to $w_1$ and $w_2$ for $i\in[3]$.

We will construct $H^\ast$ in the following steps.
We first add $w_1,w_2$ and $v_1w_1,xw_2$ to $Z_{k'}$ and obtain $Z_{k'+2}$. For $i\in[2]$, add $u_i$ and $u_iw_1,u_iw_2,u_iz_{i}$ to $Z_{k'+2}$ and obtain $Z_{k'+4}$, and add $u_3$ and $u_3w_1,u_3w_2,u_3z_1$ to $Z_{k'+4}$ to obtain $Z_{k'+5}$. In the last step, we add remaining vertices in $T-T'$ one by one by joining it with two neighbors in $V(Z_{k'+5})$ to get $H^\ast$.  Thus, we have that $$|E(H^\ast)|=|E(Z_{k'})|+2+3\cdot3+2(|T-T'|-3)=2|V(G)|-3.$$ Let $W=U'\cup\{w_1,w_2,u_1,u_2,u_3\}$. It is easy to check that for any subset $W'\subseteq W$ and any subgraph $Z'\subseteq Z_k$, $|E_{H^\ast}(W'\cup V(Z'))|-|E(Z')|\le 2|W'|.$ 

It is remaining to show that $|E(H')|\le 2|V(H')|-3$ for any subgraph $H'\subseteq H^\ast$ with $|H'|\ge2$. Clearly, if $H'\subseteq Z_k$, then $|E(H')|\le 2|V(H')|-3$. If $V(H')\subseteq V(H^\ast)-V(Z_k)$, then
\begin{align*}
|E(H')|&\le |V(H')\cap(S-S')|\cdot |V(H')\cap(T-T')|\\
&+|E[U'\cap V(H'),(S-S')\cap V(H')]|+|E(U'\cap V(H'))|\\
&\le\max\{|V(H')|-1,2(|V(H')|-2),(|V(H')|-2)+1,2(|V(H')-3|)+2,\\&(|V(H')|-3)+1+1,2(|V(H')|-4)+2+1\}\\
&\le2|V(H')|-3.
\end{align*}
Next, we consider that $V(H_1')=V(H')\cap V(Z_k)\ne\emptyset$ and $V(H'_2)=V(H')-V(H'_1)\ne\emptyset$. Each vertex in $V(H^\ast)-V(Z_k)-W$ is a 2-vertex in $H^\ast$. Thus, every time we add any subset $Y\subseteq V(H^\ast)-V(Z_k)$ into $Z_k$,  we have that
$|E_{H^\ast}(Y\cup V(Z'))|-|E(Z')|\le 2|Y|$ for any subgraph $Z'\subseteq Z_k$. Thus, $|E(H')|\le |E(H_1')|+2|V(H_2')|\le2|V(H')|-3$, completing the proof.

\medskip
\textbf{Subcase 1.3:} $|S-S'|=3$.

Then $s=6$ and $d(v_1)=5$ since $|S'|\ge|V(Z)\cap S|\ge3$ by Conclusion \ref{con2}. Thus, $|S'|=|V(Z)\cap S|=3$ and $|S'\cap \{v_2,v_3\}|=2\ \&\  |S'\cap N'(v_1)|=1$ or $|S'\cap N'(v_1)|=2\ \&\  |S'\cap \{v_2,v_3\}|=1$. As $|N(x)\cap Z|\le2$ and $|N(v_1)\cap Z|\le2$, $x,v_1\notin V(Z)$. Note that  $|V(Z)|\ge4$, so $T\cap V(Z)\ne\emptyset$. It follows that every vertex in $T\cap V(Z)$ is adjacent to all vertices in $S'$.

We show that $|T-T'|\ge4$. Suppose to the contrary that $|T-T'|\le3$. If $U\subseteq V(Z_k)$, then at most one vertex say $w$ in $S-S'$ has another neighbor in $Z_k-U$. If $U\nsubseteq V(Z_k)$, then $|S_1|\le1$ by Conclusion \ref{con4}. So there must exist a vertex say $w_1\in N'(v_1)\cap(S-S')$ such that $N(w_1)\cap(Z_k-U)=\emptyset$. Suppose that $v_2\in S'$. Then $C=(T-T')\cup(S-w_1-v_2)\cup x$ is an essential cut of size at most 8, as $\{v_1,w_1\}$ and $v_2\cup T'$ belong to different nontrivial components, a contradiction.

Let $N'(z_i)=N(z_i)\cap (T-T')$ for $z_i\in S'$ and $i\in[3]$. We show that $|N'(z_i)|\ge3$ for any $i\in[3]$. Suppose that for some $i\in[3]$, say $i=1$, $|N'(z_1)|\le2$. If $U\subseteq V(Z_k)$, then at most one vertex say $w$ in $S-S'$ has another neighbor in $Z_k-U$. If $U\nsubseteq V(Z_k)$, then $|S_0|\ge2$ by Conclusion \ref{con4}. So there must exist two vertices, say $w_1,w_2\in S-S'$ such that $N(w_i)\cap(Z_k-U)=\emptyset$ for $i\in[2]$. Then $C=\{x,v_1\}\cup (S-w_1-w_2-z_1)\cup N'(z_1)$ is an essential cut of size at most 7, as $(T-T'-N'(z_1))\cup\{w_1,w_2\}$ and $z_1\cup T'$ belong to different nontrivial components, a contradiction. Since all but at most one vertex  in $T-T'$ can be only adjacent at most one vertex in $S'$, we have that $|T-T'|\ge8$ and there are three different vertices $u_i\in T-T'$ for $i\in[3]$ and $w_1,w_2\in S-S'$ such that in $G-e$, $u_i$ has a neighbor in $S'$ and is adjacent to $w_1,w_2$.

(\romannumeral1) Suppose that $U\subseteq V(Z_k)$.

We add $w_i$ and an edge in $E[w_i,U]$ to $Z_k$ and obtain $Z_{k+2}$, where $i\in[2]$. Then add $u_i$ and three incident edges including $u_iw_1,u_iw_2$ and an edge in $E[u_i,S']$ to $Z_{k+2}$ and obtain $Z_{k+4}$ for $i\in[2]$.
If $e\notin E(Z_{k+4})$, then add the rest vertices of $T-T'$ and $w_3$ one by one by joining it with two neighbors in the already-formed subgraph in $G-e$ to get $H^\ast$.
If $e\in E(Z_{k+4})$, then add $u_3$ and three incident edges including $u_3w_1,u_3w_2$ and an edge in $E[u_3,S']$ to $Z_{k+4}$ and obtain $Z_{k+5}$. Add the rest vertices of $T-T'$ and $w_3$ one by one by joining it with two neighbors in the already-formed subgraph and delete $e$ to get $H^\ast$.
In both cases,
$$|E(H^\ast)|=|E(Z_{k})|+2|V(G)-V(Z_k)|=2|V(G)|-3.$$

At last, we will check that $|E(H')|\le 2|V(H')|-3$ for any subgraph $H'\subseteq H^\ast$ with $|H'|\ge2$.
Clearly, if $H'\subseteq Z_k$, then $|E(H')|\le 2|V(H')|-3$.  If $V(H')\subseteq V(H^\ast)-V(Z_k)$,
\begin{align*}
|E(H')|&\le |V(H')\cap(S-S')|\cdot |V(H')\cap(T-T')|+|E(S-S')|\\
&\le\max\{|V(H')|-1,2(|V(H')|-2)+1,2(|V(H')|-3)+2\}\\
&\le2|V(H')|-3.
\end{align*}
Next, we consider that $V(H_1')=V(H')\cap V(Z_k)\ne\emptyset$ and $V(H'_2)=V(H')-V(H'_1)\ne\emptyset$. Let $W=\{w_1,w_2,u_1,u_2,u_3\}$. It is easy to check that for any subset $W'\subseteq W$ and any subgraph $Z'\subseteq Z_k$, $$|E_{H^\ast}(W'\cup V(Z'))|-|E(Z')|\le 2|W'|.$$ Since we add all vertices in $V(G)-V(Z_k)-W$ one by one by joining it with two neighbors in the already-formed subgraph, we have that for any subset $Y\subseteq V(H^\ast)-V(Z_k)$,  $$|E_{H^\ast}(Y\cup V(Z'))|-|E(Z')|\le 2|Y|,$$ for any subgraph $Z'\subseteq Z_k$. Thus, $|E(H')|\le |E(H_1')|+2|V(H_2')|\le2|V(H')|-3$.

(\romannumeral2) Suppose that $U\nsubseteq V(Z_k)$.

We add $w_i$ and an edge in $E[w_i,U]$ to $Z_{k'}$ and obtain $Z_{k'+2}$, where $i\in[2]$. Then for $i\in[3]$, add $u_i$ and three incident edges including $u_iw_1,u_iw_2$ and an edge in $E[u_i,S']$ to $Z_{k'+2}$ and obtain $Z_{k'+5}$. Add the rest vertices of $T-T'$ and $w_3$ one by one by joining it with two neighbors in the already-formed subgraph to get $H^\ast$. Clearly,  $$|E(H^\ast)|=|E(Z_{k})|+2|V(G)-V(Z_k)|=2|V(G)|-3.$$

At last, we will check that $|E(H')|\le 2|V(H')|-3$ for any subgraph $H'\subseteq H^\ast$ with $|H'|\ge2$. Clearly, if $H'\subseteq Z_k$, then $|E(H')|\le 2|V(H')|-3$.
Suppose that $V(H')\subseteq V(H^\ast)-V(Z_k)$. If $U'\cap V(H')=\emptyset$, then
\begin{align*}
|E(H')|&\le |V(H')\cap(S-S')|\cdot |V(H')\cap(T-T')|+|E(S-S')|\\
&\le\max\{|V(H')|-1,2(|V(H')|-2)+1,2(|V(H')|-3)+2\}\\
&\le2|V(H')|-3.
\end{align*}
If $|U'\cap V(H')|=1$, then
\begin{align*}
|E(H')|&\le |V(H')\cap(S-S')|\cdot |V(H')\cap(T-T')|+|E(S-S')|\\
&+|E[U'\cap V(H'), (S-S')\cap V(H')]|\\
&\le\max\{(|V(H')|-2)+1,2(|V(H')|-3)+1+2,2(|V(H')|-4)+2+3 \}\\
&\le2|V(H')|-3.
\end{align*}
If $|U'\cap V(H')|=2$, then
\begin{align*}
|E(H')|&\le |V(H')\cap(S-S')|\cdot |V(H')\cap(T-T')|+|E(S-S')|\\
&+|E[U'\cap V(H'), (S-S')\cap V(H')]|+|E(U')|\\
&\le\max\{(|V(H')|-3)+1,2(|V(H')|-4)+1+2,2(|V(H')|-5)+2+3 \}+1\\
&\le2|V(H')|-3.
\end{align*}
Next, we consider that $V(H_1')=V(H')\cap V(Z_k)\ne\emptyset$ and $V(H'_2)=V(H')-V(H'_1)\ne\emptyset$. Let $W=U'\cup\{w_1,w_2,u_1,u_2,u_3\}$. It is easy to check that for any subset $W'\subseteq W$ and any subgraph $Z'\subseteq Z_k$, $$|E_{H^\ast}(W'\cup V(Z'))|-|E(Z')|\le 2|W'|.$$
Since we add all vertices in $V(G)-V(Z_k)-W$ one by one by joining it with two neighbors in the already-formed subgraph, we have that for any subset $Y\subseteq V(H^\ast)-V(Z_k)$,  $$|E_{H^\ast}(Y\cup V(Z'))|-|E(Z')|\le 2|Y|,$$ for any subgraph $Z'\subseteq Z_k$. Thus, $|E(H')|\le |E(H_1')|+2|V(H_2')|\le2|V(H')|-3$.

\noindent\textbf{Case 2:} $S'\subseteq N'(v_1)$.

In this case, we have that $|N'(v_1)\cap V(Z)|\ge3$ by Conclusion \ref{con2}. Let $$T''=\{u\in T: |N_{G-e}(u)\cap V(Z)|\ge2 \}.$$
In $G-e$, we add the rest vertices of $v_1\cup T''$ to $Z_0$ one by one by joining it with two neighbors in $Z_0$ and get $Z_{h}$.
Note that $V(Z_h)\subseteq V(Z_k)$ and $v_2,v_3\notin V(Z_h)$. As $x$ has at most one neighbor in $V(Z)$, $x\notin V(Z)$ by the minimality of the $|V(Z)|$. Since $Z\subseteq H$, $S\cap V(Z)$ is independent in $H$. We also know that $v_1\cup T$ is an independent set. Therefore, $Z$ is a bipartite graph and $|V(Z)|\ge7$. It implies that $|T''|\ge|Z\cap T|\ge7-5=2$. Since $v_2,v_3\notin V(Z_k)$, there is a vertex, say $v_2$ that is adjacent to at most one vertex in $Z_k$. It follows that $v_2$ is adjacent to at most one vertex in $Z_h$. Let $Y=N(v_2)\cap Z_h$. Then $|Y|\le1$. 
We will show that $|T-T''|\ge5$. Consider $C=v_1\cup(S-V(Z)-v_2)\cup(T-T'')\cup Y$. Note that $\{x,v_2\}$ and $(T''\cup (V(Z)\cap S))-Y$ belong to different components of $G-C$. As $|T''|\ge2$ and $|V(Z)\cap S|\ge3$, there must be an edge in $(T''\cup (V(Z)\cap S))-Y$. This implies that $C$ is an essential cut. Since $G$ is essentially $9$-connected, $|C|\ge9$ and $|T-T''|\ge5$.

We add $x$ and $xv_1$ to $Z_h$ and add all vertices in $S-V(Z)$ one by one by joining it with a neighbor in $U$ and get $Z_{h'}$. Let $\ell=|S-V(Z)|+1$. Note that $|T-T''|\ge5\ge\ell+1$. Let $$W=\{u_i:u_i\in T-T''\  \&\  u_i\  is\  not\  incident\  with\  e\}.$$
If $e\in Z_{h'}$, then $e$ is not incident with any vertex in $T-T''$ and $|W|=|T-T''|\ge\ell+1$. Let $W'\subseteq W$ and $W'=\{u_1,\ldots,u_{\ell+1}\}$. Then add all vertices in $W'$ one by one by joining it with three neighbors in $Z_{h'}$ to get $Z_{h'+\ell+1}$.  At last, add the rest vertices of $T-T''$ one by one by joining it with two neighbors in the already-formed subgraph and delete $e$ to get $H^\ast$. If $e\notin Z_{h'}$, then $|W|\ge|T-T''|-1\ge\ell$. Let $W'\subseteq W$ and $W'=\{u_1,\ldots,u_{\ell}\}$. Then add all vertices in $W'$ one by one by joining it with three neighbors in $Z_{h'}$ and get $Z_{h'+\ell}$. At last, add the rest vertices of $T-T''$ one by one by joining it with two neighbors in the already-formed subgraph in $G-e$ and get $H^\ast$. Clearly, in both cases, $$|E(H^\ast)|=|E(Z_{h})|+2|V(G)-V(Z_h)|=2|V(G)|-3.$$

For any subgraph $H'\subseteq H^\ast$, if $|E(H')|\le 2|V(H')|-3$, then we are done. Otherwise, there is a minimal subgraph $\overline{Z}\subseteq H^\ast$ such that $|E(\overline{Z})|> 2|V(\overline{Z})|-3$. By the process of adding vertices and edges, there is some $u_{j_0}\in W'$ such that $u_{j_0}\in V(\overline{Z})$ and $N_{H^\ast}(u_{j_0})\subseteq V(\overline{Z})$. As in $H^\ast$ every vertex in $W'$ has at least two neighbors in $S-V(Z)$ and must have a neighbor in $\{v_2,v_3\}$, we have that $\{v_2,v_3\}\cap V(\overline{Z})\ne\emptyset$. Note that $\overline{Z}\subseteq H^\ast\subseteq H$. Thus, we can regard $\overline{Z}$ as a new $Z$. Then we get $\overline{Z_0}$ and $\overline{Z_k}$ in the same way. Let $\overline{S'}=S\cap \overline{Z}$. Note that $\overline{Z}\nsubseteq N'(v_1)$ and so $\overline{S'}\nsubseteq N'(v_1)$, which is same as Case 1.
\end{proof}


\begin{thebibliography}{99}

\bibitem{AsRo78}
L. Asimov, B. Roth, The rigidity of graphs, Trans. Amer. Math. Soc. 245 (1978) 279-289.

\bibitem{BoMu08}
J. A. Bondy and U. S. R. Murty, {\it Graph Theory}, Springer, New York, 2008.


\bibitem{Conn05}
R. Connelly, Generic global rigidity, Discrete Comput. Geom., 33 (2005) 549-563.







\bibitem{Hend92}
B. Hendrickson, Conditions for unique graph realizations, SIAM J. Comput. 21 (1992) 65-84.

\bibitem{JaJo05}
B. Jackson and T. Jord\'an, Connected rigidity matroids and unique realizations of graphs, J. Combin. Theory Ser. B 94 (2005) 1-29.

\bibitem{JaJo09}
 B. Jackson and T. Jord\'an, A sufficient connectivity condition for generic rigidity in the plane, Discrete Appl. Math. 157 (2009) 1965-1968.

\bibitem{JaSS07}
B. Jackson, B. Servatius and H. Servatius, The 2-dimensional rigidity of certain families of graphs, J. Graph Theory 54 (2007) 154-166.

\bibitem{JaWo92}
B. Jackson and N. Wormald, Longest cycles in 3-connected planar graphs, J. Combin. Theory Ser. B 54 (1992) 291-321.

\bibitem{LSWZ06}
H.-J. Lai, Y. Shao, H. Wu and J. Zhou, Every 3-connected, essentially 11-connected line graph is Hamiltonian,
J. Combin. Theory Ser. B 96 (2006) 571-576.

\bibitem{Lama70}
G. Laman, On graphs and rigidity of plane skeletal structures, J. Engrg. Math. 4 (1970) 331-340.

\bibitem{LoYe82}
L. Lov\'asz and Y. Yemini, On generic rigidity in the plane, SIAM J. Algebr. Discrete Methods 3 (1982) 91-98.


\bibitem{Poll1927}
H. Pollaczek-Geiringer, \"{U}ber die Gliederung ebener Fachwerke. Z. Angew. Math. Mech. 7:1 (1927) 58-72.

\bibitem{Poll1932}
H. Pollaczek-Geiringer, Zur Gliederungstheorie R\"aumlicher Fachwerke. Z. Angew. Math. Mech. 12:6 (1932) 369-376.

\bibitem{Whit96}
W. Whiteley, Some matroids from discrete applied geometry, Contemp. Math. 197 (1996) 171-312.



\end{thebibliography}
\end{document}